%% file: main.tex
\documentclass{amsart}
\usepackage[utf8]{inputenc}
\usepackage[top=1in,left=1in,right=1in,bottom=1in]{geometry}
\usepackage{booktabs}
\usepackage{todonotes}
\usepackage{longtable}
\usepackage{longtable} 
\usepackage{amsmath, amssymb}
\usepackage{amsthm}
\usepackage{amsfonts,mathrsfs}
\usepackage{amscd}
\usepackage{array}
\usepackage{bigstrut}
\usepackage{color}
\usepackage{listings}
\usetikzlibrary{decorations.pathreplacing}
\usetikzlibrary{matrix,arrows.meta}
\usepackage[center]{caption}
\usepackage{enumitem}
\usepackage{pgfplots}
\usepackage{subfigure}
\usepackage{hyperref}

\usepackage{flafter}

\usepackage{subfigure}
\usepackage{wrapfig}
\usepackage{tikz-cd}
\usepackage{graphicx}
\usepackage{caption}
\newcolumntype{L}{>{\centering\arraybackslash}p{3cm}}

\usepackage{pgf,tikz}
\usetikzlibrary{arrows}
\usetikzlibrary{shapes,decorations}
\usepackage{tkz-graph}

\newtheorem{theorem}{Theorem}[section]
\newtheorem*{theorem*}{Theorem}
\newtheorem{lemma}[theorem]{Lemma}

\newtheorem{proposition}[theorem]{Proposition}

\theoremstyle{definition}
\newtheorem{definition}[theorem]{Definition}

\theoremstyle{remark}

\newtheorem{example}[theorem]{Example}

\newcommand{\PP}{\mathbb{P}}
\newcommand{\CC}{\mathbb{C}}

\newcommand{\ZZ}{\mathbb{Z}}
\newcommand{\QQ}{\mathbb{Q}}

\newcommand{\NN}{\mathbb{N}}
\renewcommand{\AA}{\mathbb{A}}
\newcommand{\cE}{\mathcal{E}}
\newcommand{\cK}{\mathcal{K}}
\newcommand{\cB}{\mathcal{B}}

\newcommand{\cX}{\mathcal{X}}
\newcommand{\cZ}{\mathcal{Z}}

\newcommand{\set}[1]{\left\{#1\right\}}  

\DeclareMathOperator{\Hom}{Hom}  
\DeclareMathOperator{\Aut}{Aut}  
\DeclareMathOperator{\SL}{SL}   
\DeclareMathOperator{\Pic}{Pic}  
\DeclareMathOperator{\rk}{rank}  
\DeclareMathOperator{\sign}{sign}
\DeclareMathOperator{\GL}{GL}

\DeclareMathOperator{\spec}{sp}

\usepackage{epsfig}

\makeatletter
\def\imod#1{\allowbreak\mkern10mu({\operator@font mod}\,\,#1)}
\makeatother

\definecolor{pistachio}{rgb}{0.58, 0.77, 0.45}
\definecolor{red(munsell)}{rgb}{0.95, 0.0, 0.24}
\definecolor{eggshell}{rgb}{0.94, 0.92, 0.84}

\title{Mirror symmetry for K3 surfaces}
\author{C.J. Bott}
\address{Department of Mathematics, 219E BLOC, Texas A\&M University, College Station, TX, 77840}
\email{cbott2@math.tamu.edu}

\author{Paola Comparin}
\address{
Departamento de Matem\'atica y Estad\'istica, 
Universidad de La Frontera, 
Av. Francisco Salazar 1145, Temuco, Chile}
\email{paola.comparin@ufrontera.cl}

\author{Nathan Priddis}
\address{Department of Mathematics, 
275 TMCB, Brigham Young University, 
Provo, UT 84602, USA.}
\email{priddis@mathematics.byu.edu}

\begin{document}

\begin{abstract}
For certain K3 surfaces, there are two constructions of mirror symmetry that are very different. The first, known as BHK mirror symmetry, comes from the Landau--Ginzburg model for the K3 surface; the other, known as LPK3 mirror symmetry, is based on a lattice polarization of the K3 surface in the sense of Dolgachev's definition. There is a large class of K3 surfaces for which both versions of mirror symmetry apply. In this class we consider the K3 surfaces admitting a certain purely nonsymplectic automorphism of order 4, 8, or 12, and we complete the proof that these two formulations of mirror symmetry agree for this class of K3 surfaces.

\end{abstract}
\maketitle

\section*{Introduction}

The phenomenon of mirror symmetry, first arising in physics in the context of string theory, started interesting mathematicians in the 90's, when a conjecture by a group of physicists \cite{candelas} provided a method for
counting rational curves on the quintic threefold that was much more effective then all formerly known methods. Since that time, mathematicians have expended much energy trying to formulate and understand this phenomenon mathematically.  
There are several mathematical constructions of mirror symmetry in various contexts.

The idea of mirror symmetry is a correspondence between families of {\em Calabi-Yau varieties}, which essentially trades the information of complex structures on the first family with K\"ahler structures on the second. 
For example, a first prediction of mirror symmetry is 
a relationship between the Hodge diamonds of these varieties.  
For general members $X$ and $X'$ of two mirror families of Calabi-Yau varieties $\mathcal F$ and $\mathcal F'$ we expect the following relationship between the Hodge numbers.  
\[
h^{p,q}(X)=h^{n-p,q}(X')
\]
where $n$ is the dimension of $X,X'$.

Various authors have found mathematical constructions for mirror symmetry in various contexts, including Batyrev-Borisov \cite{batyrev}, Givental \cite{givental}, Hori-Vafa \cite{HV}, and others. 
In this article we are interested in two particular formulations of mirror symmetry, namely BHK mirror symmetry and mirror symmetry for lattice polarized K3 surfaces, which we now describe. 

BHK mirrror symmetry was proposed by Berglund--H\"ubsch \cite{berghub}, Berglund--Henningson \cite{berghenn}, and later Krawitz \cite{krawitz}. 
This construction applies to Calabi-Yau varieties in any dimension, given as hypersurfaces in weighted projective spaces $\mathbb P(w_0,\ldots,w_{m})$. In order to construct the Calabi-Yau variety $X_{W,G}$, one needs to choose a quasismooth invertible polynomial in $\mathbb P(w_0,\ldots,w_{m})$ of degree $d=\sum_{i=0}^{m}w_i$ and a subgroup $G$ of diagonal 
automorphisms of the vanishing locus of $W$ satisfying certain conditions (more details can be found in Section~\ref{sec:BHK}).
This version of mirror symmetry provides a rule for constructing a dual polynomial $W^T$ and a dual group $G^T$ from which one can construct the \emph{BHK mirror} $X_{W^T,G^T}$. This rule has been proven to be mirror in the sense of Hodge numbers by Chiodo and Ruan in \cite{BHCR}. Others have use BHK mirror symmetry to prove deeper relationships between the mirror manifolds, 
e.g. \cite{CIR}, \cite{ChR},  \cite{clader}, \cite{guere}, \cite{PSh}. 

When passing to Calabi-Yau varieties of dimension 2, i.e. \emph{K3 surfaces}, since the Hodge diamond of a K3 surface is fixed, a specific definition of mirror symmetry specific to K3 surfaces has been introduced by Dolgachev, Voisin and others, which we will call LPK3 mirror symmetry. This version of mirror symmetry 
involves the behaviour of a lattice primitively embedded in the Picard group $\Pic(X)$. In other words, one begins with a lattice polarization of a K3 surface, and defines the LPK3 mirror family as the family of K3 surfaces polarized by a certain \emph{mirror lattice} (see \cite{dolgachev} and Section \ref{ss:LPK3mirror} for details). This is the second formulation of mirror symmetry that we consider. 

A natural question that arises in this context is whether these two definitions of mirror symmetry agree, i.e. if the mirror K3 surfaces obtained via the BHK construction belong to LPK3 mirror families or not. This question has been partially answered; in the current paper we will complete the proof, by answering the question in the affirmative for a class of K3 surfaces admitting a non--symplectic automorphism of order $n=4,8,12$.

More specifically, we begin with a K3 surface $X_{W,G}$ obtained from a pair $(W,G)$ of a polynomial of the form $W=x_0^n+f(x_1,x_2,x_3)$ and a group satisfying certain conditions (see Section~\ref{sec:BHK} for the construction). This K3 surface naturally admits a non--symplectic automorphism $\sigma_n(x_0,x_1,x_2,x_3)=(\zeta_n x_0,x_1,x_2,x_3)$, where $\zeta_n$ is a primitve $n$-th root of unity. The invariant lattice 
\[
S(\sigma_n):=\{x\in H^2(X_{W,G}, \ZZ): \sigma_n^*x=x\}
\]
polarizes the K3 surface $X_{W,G}$. The same happens in the dual context for the BHK mirror $X_{W^T,G^T}$ with $\sigma_n^T$ and $S(\sigma^T_n)$ defined in the same way. In order to show the equivalence of BHK and LPK3 mirror contructions, we prove the following theorem (see Section~\ref{s:MainResult}).

\begin{theorem*}
Given an invertible polynomial of the  form $W=x_0^n+f(x_1,x_2,x_3)$ quasihomogeneous with respect to one of Reid and Yonemura’s 95 weight systems, and $G$ a group of symmetries of $W$ satisfying $J_W\leq G\leq \SL_W$. 
Then the K3 surface $X_{W,G}$, polarized by $S(\sigma_n)$, and its BHK mirror $X_{W^T,G^T}$ polarized by $S(\sigma_n^T)$, form an LPK3 mirror pair.
\end{theorem*}

This result has been proved for $n=2$ by Artebani, Boissi\`ere and Sarti in \cite{ABS} and for $n$ prime by Comparin, Lyons, Priddis and Suggs in \cite{CLPS}, and for all other $n$, except for $n=4,8,12$ in \cite{CP}. In all cases, the proof of the theorem is done computing the invariant lattices $S(\sigma_n)$ and $S(\sigma_n^T)$ for the K3 surfaces $X_{W,G}$ and $X_{W^T, G^T}$ and then comparing them, in order to show they are mirror lattices.
In the first two papers the computations rely heavily on the relations between topological invariants of the fixed locus of the automorphism $\sigma_n$ and the lattice invariants of $S(\sigma_n)$, when $n$ is prime. 

Such relations are no longer avalaible when $n$ is not prime, so that in \cite{CP} the authors introduce new methods for computing $S(\sigma_n)$ in order to prove the theorem for $n$ composite and different from $4,8,12$. 
However, these methods we not sufficient for $n=4,8,12$. 

In the current paper we introduce new methods, in order to study the missing cases and complete the proof of the theorem for any $n$. The new methods include using certain isomorphisms and deformations in order to group K3 surfaces into equivalence classes with the property that every K3 surface in a given equivalence class has the same invariant lattice $S(\sigma_n)$. 
Then one can apply a modification of former methods to obtain the required information about $S(\sigma_n)$ needed to prove the theorem. In some cases, a deeper analysis is needed, so we exploit knowledge of lines on certain Fermat surfaces in weighted projective space to complete the analysis.

The paper is organized as follows: in Section \ref{sec:back} we recall some preliminary results on lattices and K3 surfaces and we present in Sections \ref{ss:LPK3mirror} and \ref{sec:BHK} the two definitions of LPK3 and BHK mirror symmetry.
In Section \ref{s:MainResult} we state the main theorem and show how isomorphisms and deformations can be used in the proof of it.
Section \ref{s:proofMain} is devoted to the calculations which prove the theorem, while in Section \ref{sec:exc} we deal with the exceptional cases that need special methods.
Appendix \ref{app:tables} contains tables with the calculations proving the theorem.

\section*{Acknowledgments}
We would like to thank Michela Artebani, Alessandra Sarti and Matthias Sch\"utt
for many useful discussions and helpful insights. 
The second author has been partially supported by Proyecto Fondecyt Postdoctorado N. 3150015 and Proyecto Anillo ACT 1415 PIA Conicyt.

\section{Background}\label{sec:back}

In this section, in order to set notation, we will give the necessary background material. 
As this can be found in several other places, e.g. \cite{ABS,CLPS, nikulin}, we will be terse.  

\subsection{K3 surfaces and lattices}

A \emph{K3 surface} $X$ is a compact complex surface  having trivial canonical bundle  and $h^{1,0}(X)=0$.
Let $\omega_X$ be the  basis of the 1-dimensional space $H^{2,0}(X)$.
Given an automorphism $\sigma\in\Aut(X)$, we call $\sigma$ \emph{symplectic} if $\sigma^*\omega_X=\omega_X$ and \emph{non-symplectic} otherwise. If $\sigma$ has order $n$ and  $\sigma^*\omega_X=\zeta_n\omega_X$ with $\zeta_n$ a primitive $n$-th root of unity, $\sigma$ is called \emph{purely non-symplectic of order $n$}.

An \emph{(integral) lattice} is defined to be a pair $(L,B)$, where $L$ is a free abelian group of finite rank and $B:L\times L \to \ZZ$ is a non-degenerate symmetric bilinear form. A lattice is called \emph{even} if $B(x,x) \in 2\ZZ$ for all $x \in L$. 
We will denote the \emph{signature} of the lattice $(L,B)$ by $(l_+,l_-)$. A lattice is \emph{hyperbolic} if $l_+=1$. 
From now on, we denote a lattice $(L,B)$ by $L$ for convenience, making $B$ explicit when required.

A \emph{sublattice} $(L',B') \subseteq (L,B)$ is a free abelian subgroup $L' \subseteq L$, where $B$ restricted to $L'$ is $B'$. A sublattice $L' \subseteq L$ is \emph{primitively embedded} if $L/L'$ is free. Furthermore $L$ is called an \emph{overlattice} of $L'$ if $L/L'$ is a finite abelian group. 

The following are some lattices of particular interest: 
\begin{enumerate}
\item The rank 2 hyberbolic lattice whose bilinear form is given by the matrix 
$\left(\begin{matrix} 
0&1 \\
1&0 \end{matrix}\right)$ is called $U$.
\item  The lattices $A_l$, $D_m$, and $E_n$ ($l\geq 1$,$m\geq 4$, $6\leq n\leq 8$) are negative definite lattices whose bilinear form is given by the adjacency matrices for the Dynkin diagrams of the classic $ADE$-singularities.
\item The $T_{p,q,r}$ lattice has rank $p+q+r-2$ and is defined by the adjacency matrix for a graph in the form of a $T$ with $p,q,r$ the respective lengths of the legs (see Figure~\ref{T_{4,4,4}} for $T_{4,4,4}$).
\item For $n\in\ZZ$, the rank $1$ lattice $\langle n\rangle$ is defined by multiplication by $n$, i.e. $B(x,y)=xny$ for $x,y\in L\cong \ZZ$.
\item If $L$ is a lattice and $n\in\ZZ$, we will write $L(n)$ to express the lattice whose values are $n$ times those for $L$.
\end{enumerate}

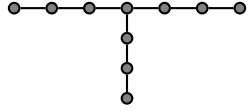
\begin{figure}[h]
\centering
\begin{tikzpicture}[xscale=.5,yscale=.4, thick, every node/.style={circle, draw, fill=black!50, inner sep=0pt, minimum width=4pt}]
  \node (n1) at (-3,0) {};
  \node (n2) at (-2,0) {};
  \node (n3) at (-1,0) {};
  \node (n4) at (0,0) {};
  \node (n5) at (0,-1) {};
  \node (n6) at (0,-2) {};
  \node (n7) at (0,-3) {};
  \node (n8) at (1,0) {};
  \node (n9) at (2,0) {};
  \node (n10) at (3,0) {};

  \foreach \from/\to in {n1/n2,n2/n3,n3/n4,n4/n5,n5/n6,n6/n7,n4/n8,n8/n9,n9/n10}
    \draw (\from) -- (\to);

\end{tikzpicture}
\caption{The graph for the lattice $T_{4,4,4}$} \label{T_{4,4,4}}
\end{figure}

Given a lattice $L$, we denote $L^*=\Hom(L,\ZZ)$, and define the \emph{discriminant group} $A_L:=L^*/L$, which is a finite abelian group. In fact if we write $B$ as a matrix, then $|A_L|=|\det(B)|$. If $A_L$ is trivial (i.e. $L=L^*$), we call the lattice $L$ \emph{unimodular}. 

Given a finite abelian group $A$, a \emph{finite quadratic form} is a map $q:A \to \QQ/2\ZZ$ satisfying the following two conditions:
\begin{enumerate}
    \item For all $n \in \ZZ$ and $a \in A$, $q(na)=n^2q(a)$, and 
    \item There exists a symmetric, bilinear form $b:A\times A \to \QQ/\ZZ$ such that for all $a,a' \in A$, 
    \[
    q(a+a') \equiv q(a) + q(a') + 2b(a,a') \mod{2\ZZ}.
    \]
\end{enumerate}

If $L$ is a lattice, we can extend the bilinear form $B$ on $L$ to $L^*$ (now taking values in $\QQ$). If $L$ is even, we get an induced finite quadratic from $q_L: A_L \to \QQ/2\ZZ$, which is called the \emph{discriminant quadratic form of $L$}.

If $M \subseteq L$ is a sublattice, we denote the orthogonal complement in $L$ as 
\[
M_L^{\perp}:=\{l \in L : B(l,m)=0 \text{ for all }m \in M\}.
\]
Two lattices $M$ and $K$ are said to be \emph{orthogonal} if there exists an even, unimodular lattice $L$ such that $M \subseteq L$ and $M_L^{\perp} \cong K$.
Orthogonality will be key in our later definition of LPK3 mirror symmetry and the following lemma provides a useful criterium to check orthogonality.

\begin{lemma}{\rm (\cite[Corollary 1.6.2]{nikulin})}\label{lem:orthogonal}
The lattices $L$ and $K$ are orthogonal if and only if $q_L\cong -q_K$.
\end{lemma}

The set of all finite quadratic forms is a semi-group under direct sum. There are three classes of quadratic forms which generate the semi-group of all finite quadratic forms under direct sum, which we present now following \cite{belcastro, nikulin}, :
\begin{enumerate}
\item For $p\neq2$ prime, $k\in\NN$, and $\epsilon\in\{-1,1\}$, let $a$ be the smallest even integer that has $\epsilon$ as quadratic residue modulo $p$. Then we define  
\[
w_{p,k}^{\epsilon}:\ZZ/p^k\ZZ\to\QQ/2\ZZ\ \mbox{ by }\ w_{p,k}^{\epsilon}(1)=\frac{a}{p^k}.
\]

\item For $k\in\NN$, and $\epsilon\in\{-1,1,-5,5\}$, we define 
\[
w_{2,k}^{\epsilon}:\ZZ/2^k\ZZ\to\QQ/2\ZZ\ \mbox{ by }\ w_{2,k}^\epsilon(1)=\frac{\epsilon}{2^k}.
\]

\item For $k\in\NN$, we define the quadratic forms $u_k$ and $v_k$ on $\ZZ/2^k\ZZ \times \ZZ/2^k\ZZ$ by 
\[
u_k=\begin{bmatrix}
    0&\frac{1}{2^k}\\
    \frac{1}{2^k}&0 
    \end{bmatrix}, \qquad v_k=\frac{1}{2^k}\begin{bmatrix}
    2&1 \\
    1&2
    \end{bmatrix}.
\]
\end{enumerate} 

In Table~\ref{Lattices and Quadratic Forms} we list the lattices relevant to this work, together with their signatures and their associated quadratic forms. 

\begin{table}[h!]\centering
\begin{tabular}{l c c}
  Lattice&Signature&Form\\
  \hline 
  $U$ &(1,1)& trivial \\
  \hline
  $U(2)$		&(1,1)	& $u$  \\
  \hline
  $A_1$ 		&(0,1) 	& $w^{-1}_{2,1}$ \\
  \hline
  $A_2$		&(0,2)	& $w_{3,1}^{1}$ \\
  \hline
  $A_3$		&(0,3)	& $w_{2,2}^5$ \\
  \hline
  $D_4$		&(0,4)	& $v$ \\
  \hline
  $D_5$		&(0,5)	& $w_{2,2}^{-5}$ \\
  \hline
  $D_6$		&(0,6)	& $(w_{2,1}^1)^2$\\
  \hline
  $D_9$		&(0,9)	& $w_{2,2}^{-1}$ \\
  \hline
  $E_6$		&(0,6)	& $w_{3,1}^{-1}$\\
  \hline
  $E_7$		&(0,7)	& $w_{2,1}^{1}$\\
  \hline
  $E_8$		&(0,8)	& trivial\\  
  \hline
  $T_{4,4,4}$	&(1,9)	& $v_2$\\
  \hline
  $\langle4\rangle$		&(1,0)	& $w_{2,2}^1$\\
  \hline
  $\langle-4\rangle$ & (0,1) & $w_{2,2}^{-1}$ \\
  \hline
  $\langle8\rangle$		&(1,0)	& $w_{2,3}^1$\\
  \hline
  $\langle-8\rangle$ & (0,1) & $w_{2,3}^{-1}$ \\
  \hline
 \end{tabular}
 \caption{Lattices and quadratic forms} \label{Lattices and Quadratic Forms}
\end{table}

An important question regarding integral lattices is to what extent are they determined by certain invariants. The following result due to Nikulin answers this question. 

\begin{proposition}[{\cite[Cor. 1.13.3]{nikulin}}]\label{prop:existenceuniqueness}
An even lattice with signature $(l_+,l_-)$ and discriminant quadratic form $q$ exists and is unique if $l_+\geq1$, $l_-\geq1$, $l_+-l_-\equiv \sign q \pmod 8$, and $l_+ + l_- \geq 2 + l(A_q)$, where $l(A_q)$ denotes the minimum number of generators of $A_q$, and $\sign q$ denotes the signature of $q$. 
\end{proposition}

In what follows we will be identifying lattices via their discriminant quadratic forms and signatures. All but three lattices that we will consider satisfies the requirements of Proposition~\ref{prop:existenceuniqueness} and so are uniquely determined by these invariants. Two of the three exceptions have rank 1, and so are also uniquely determined by these invariants. The final exception is $U(2)$, which is the unique even indefinite lattice of rank two with discriminanat quadratic form $u$ (e.g. because it is 2-elementary). A description in terms of the lattices listed in Table~\ref{Lattices and Quadratic Forms} is easy to reconstruct.

\subsection{LPK3 mirror symmetry}\label{ss:LPK3mirror}
We now consider mirror symmetry specifically for K3 surfaces and recall Dolgachev's definition of lattice polarized K3 (LPK3) mirror symmetry, as in \cite{dolgachev}.

Let $X$ be an algebraic K3 surface and let $\sigma$ be a purely non-symplectic automorphism of $X$ of order $n$. 
It is well known that $H^2(X,\ZZ)$ is an even, unimodular lattice of signature $(3,19)$ isomorphic to $L_{K3}=U^3 \oplus (E_8)^2$. We can then look at some interesting sublattices: the {Picard lattice} of $X$ is $\Pic(X)=H^2(X,\ZZ)\cap H^{1,1}(X,\CC)$ and the {invariant lattice} of $\sigma$ is $S(\sigma) \subseteq H^2(X,\ZZ)$ given by 
\[
S(\sigma)=\{x\in H^2(X,\ZZ) : \sigma^*(x)=x\}.
\]
Observe that $S(\sigma)$ is a primitive sublattice of $\Pic(X)$ and has signature $(1,t)$ for some $t\leq19$ (see e.g. \cite{belcastro, CLPS}). 

Let $M$ be a lattice of signature $(1,t)$, $t\leq 18$. If there exists a primitive embedding $j: M\hookrightarrow \Pic(X)$, we call $X$ an \em{ $M${-polarizable K3 surface}\footnote{If we also consider the primitive embedding as part of the data, then this is what Dolgachev calls an $M$-polarized K3 surface. Since we do not consider the embedding as the part of the data, we have a somewhat coarser version.}}. \rm Observe that all lattice polarizable K3 surfaces are algebraic.
For an $M$-polarizable K3 surface $X$, the lattice $M$ naturally embeds into $L_{K3}$, leading to our final definition.

\begin{definition}\label{def:LPK3}
Let $M$ be a primitive sublattice of $L_{K3}$ of signature $(1,t)$ with $t\leq 18$ such that $(M)^{\perp}_{L_{K3}} \cong U \oplus M^\vee$. We define $M^\vee$ (up to isometry) to be the {\em mirror lattice} to $M$. Given an $M$-polarizable K3 surface $X$ and an $M'$-polarizable K3 surface $X'$ with $M'=M^\vee$ (or equivalently $M=(M')^\vee$), we say $X$ and $X'$ are {\em LPK3 mirror K3 surfaces}. 
\end{definition}

Note that there is a whole family of K3 surfaces dual to $X$, each of which is $M^\vee$-polarizable. As for duality, if $M$ is as in Definition~\ref{def:LPK3} one can check that $M^\vee$ is also primitively embedded in $L_{K3}$, has signature $(1,18-t)$, 
and that $(M^\vee)_{L_{K3}}^\perp \cong U \oplus M$, so $(M^\vee)^\vee=M$. Notice also that $q_M \cong -q_{M^\vee}$ Hence our notion of $M$-polarizable $X$ and $M^\vee$-polarizable $X'$ being mirror K3 surfaces is a duality.

\subsection{BHK mirror symmetry}\label{sec:BHK}
The second version of mirror symmetry that we will consider is called BHK mirror symmetry. This was developed by Berglund--H\"ubsch \cite{berghub}, Berglund--Henningson \cite{berghenn}, and Krawitz \cite{krawitz}. 

Let $W \in \CC[x_0,\ldots,x_n]$ be a quasihomogeneous polynomial of degree $d$ with weight system $(q_0,\ldots,q_n;d)$, i.e. $W(\lambda^{q_0}x_0,\ldots,\lambda^{q_n}x_n)=\lambda^dW(x_0,\ldots,x_n)$ for all $\lambda \in \CC^*$. In this work, we take the convention that $q_0,\dots,q_n,d\in\ZZ_{\geq 1}$ with the property that $\gcd(q_0,\dots,q_n)=1$. 

The polynomial $W\in\CC[x_0,\ldots,x_n]$ is \emph{nondegenerate} if it has a single critical point at the origin and the weights are uniquely determined. Furthermore, we say a nondegenerate quasihomogeneous polynomial is \emph{invertible} if it has the same number of monomials and variables (this last condition is also known as the Delsarte condition in the literature).

Given an invertible polynomial $W=\sum_{i=0}^n\prod_{j=0}^nx_j^{a_{ij}}\in\CC[x_0,\ldots,x_n]$, we can construct the \emph{exponent matrix} $A_W:=(a_{ij})$. The rows represent the monomials of $W$ and the columns represent the variables. The condition that $W$ be invertible, implies that $A_W$ is invertible. 

There are three building blocks of invertible polynomials, given by the following definition. 

\begin{definition}
The following types of quasihomogeneous polynomials are called \emph{atomic types}:
\begin{enumerate}
    \item \textbf{Fermat}: $W = x^n$ 
    \item \textbf{Loop}: $W = x_0x_1^{a_1} + x_1x_2^{a_2} + \ldots + x_{n-1}x_n^{a_n} + x_nx_0^{a_0}$ where $a_i\geq2$ for all $i$
    \item \textbf{Chain}: $W = x_0^{a_0} + x_0x_1^{a_1} + x_1x_2^{a_2} + \ldots + x_{n-1}x_n^{a_n}$ where $a_i\geq2$ for all $i$
\end{enumerate}
\end{definition}

As a consequence of \cite[Theorem 1]{KrSk}, a well known result for classsifying invertible polynomials is the following.

\begin{proposition}
A nondegenerate quasihomogeneous polynomial $W \in \CC[x_0,\ldots,x_n]$ is invertible if and only if it can be written as a finite sum of atomic types in disjoint sets of variables.
\end{proposition}

Let $W\in\CC[x_0,\ldots,x_n]$ be an invertible polynomial with weight system $(q_0,\ldots,q_n;d)$. We will assume from now on that the degree of $W$ equals the sum of the weights, i.e. 
\[
d=\sum_{i=0}^n q_i
\]
This is often referred to as the Calabi--Yau condition. 
This is because a quasihomogeneous polynomal of degree $d$ defines a hypersurface $Z_W$ of degree $d$ in the weighted projective space $\PP(q_0,\dots, q_n)$, and by \cite{corti-golyshev}, the Calabi--Yau condition $d=\sum_{i=0}^n q_i$ ensures that the variety $Z_W$ is a Calabi-Yau variety.

Reid (unpublished) and Yonemura \cite{yonemura} each independently showed that there are exactly 95 distinct weight systems such that the minimal resolution of $Z_W\subseteq \PP(q_0,q_1,q_2,q_3)$ yields a K3 surface. 

We now consider some groups of automorphisms of $W$.
\begin{definition}
Let $W$ be an invertible polynomial with weight system $(q_0,\ldots,q_n;d)$.

\begin{enumerate}
    \item The {\em group of diagonal symmetries} $G_W^{\max}$ of $W$ is defined by
    \[
    G_W^{max}=\{(g_0,\ldots,g_n)\in(\CC^*)^{n+1} : W(g_0x_0,\ldots,g_nx_n)=W(x_0,\ldots,x_n).\}
    \]
    
    \item Viewing the elements of $G_W^{\max}$ as diagonal matrices, the {\em special linear symmetry group} for $W$ is $\SL_W=G_W^{\max}\cap \SL_{n+1}(\CC)$.
    \item The {\em exponential grading operator group}, is denoted $J_W=\langle (e^{2\pi iq_0/d},\ldots,e^{2\pi iq_n/d})\rangle$.
\end{enumerate}
\end{definition}

Observe that each of these groups can be represented by diagonal matrices, they all are abelian and 
the entries of these matrices are all roots of unity, so $(g_0,\ldots,g_n)$=$(e^{2\pi i\cdot a_0},\ldots,e^{2\pi i\cdot a_n})$ for some $a_0,\ldots,a_n\in\QQ/\ZZ$. 
Since multiplication of roots of unity corresponds to addition of exponents, for convenience we will write these groups additively $(a_0,\ldots,a_n)\in(\QQ/\ZZ)^{n+1}$. 

With this view in mind, 
the group $\SL_W$ is the subgroup with entries that add up to an integer. 
Furthermore, if $d=\sum_{i=0}^n q_i$, the group $J_W$ is a subgroup of $\SL_W$.

The following Proposition collects two results due to Kreuzer and Krawitz \cite{kk,krawitz}. It can be found also in \cite[Section 3.1]{ABS}. 
\begin{proposition}[\cite{krawitz, kk}]
Let $W$ be an invertible polynomial with exponent matrix $A_W$ and $G_W^{\max}$ its maximal symmetry group, viewed additively.
\begin{enumerate}
    \item $|G_W^{\max}|=|\det(A_W)|$. In particular, $G_W^{\max}$ is a finite abelian group.
    \item $G_W^{\max}$ is generated by the columns of $A_W^{-1}$.
\end{enumerate}
\end{proposition} 

Let $(W,G)$ be a pair consisting of a quasihomogeneous polynomal $W$ and a group of diagonal symmetries $G\leq G^{max}_W$. From this point of view we can consider the following geometry. In order for a subgroup $G\leq G_W^{\max}$ to act on the hypersurface $Z_W$, we must have $J_W\leq G$. The equivalence relation defining points in weighted projective space is just the action of $J_W$, so $J_W$ acts trivially on $Z_W$. 

Set $\widetilde{G}=G/J_W$, and we define the variety $Z_{W,G}=Z_W/\widetilde{G}$. However, for the resulting quotient space to be a Calabi-Yau manifold, the group must preserve the canonical bundle, which means that the subgroup $G$ must also be a subgroup of $\SL_W$. When $G \neq J_W$, the group action may introduce new singularities in addition to those coming from weighted projective space. However, all of the singularities of $Z_{W,G}$ will be located on the so--called ``coordinate curves.''
 
To summarize, let $W$ be an invertible polynomial, quasihomogeneous with respect to one of Reid/Yonemura's 95 weight systems, and $G\leq G^{max}_W$ satisfying $J_W\leq G\leq \SL_W$. Then the minimal resolution of $Z_{W,G}$ is a K3 surface. We denote it $X_{W,G}$.

With this background, we can now give the BHK mirror symmetry construction. For each pair $(W,G)$ as defined above, BHK mirror symmetry produces another pair $(W^T,G^T)$. 
The dual polynomial $W^T$ is defined as the polynomial associated to the transpose of $A_W$, i.e. $A^T_W=A_{W^T}$.
With the conditions above satisfied, $W^T$ will also be an invertible polynomial (which again can be proven from the atomic type decomposition). In fact, one can check that if $W$ is quasihomogeneous with respect to one of the 95 weight systems, then $W^T$ is as well. 

Finding a dual group was a huge breakthrough in mirror symmetry, given by Berglund--Henningson \cite{berghenn} and in the Ph.D. dissertation of Krawitz \cite{krawitz}. We give the definition and some important properties below.

\begin{definition}
For an invertible polynomial $W$ with exponent matrix $A_W$ and subgroup $G \leq G_W^{\max}$, we define the {dual group} to be $G^T=\{g\in G_{W^T}^{\max} : gA_Wh^T \in \ZZ$ for all $h \in G\}$.
\end{definition}

The definition of the dual group 
does in fact have some very nice properties:

\begin{proposition}[{ \cite[Prop. 3]{ABS}}]
Let $W$ be an invertible polynomial and $G_1,G_2 \leq G_W^{\max}$.
\begin{enumerate}
    \item $(G_1^T)^T=G_1$;
    \item If $G_1 \leq G_2$, then $G_2^T \leq G_1^T$ and $G_2/G_1 \cong G_1^T/G_2^T$;
    \item $(J_W)^T=SL_{W^T}$, and $(SL_W)^T=J_{W^T}$.
\end{enumerate}\label{prop:G^T}
\end{proposition}

In particular, if $J_W \leq G\leq SL_W$, we have that $J_{W^T}\leq G^T\leq SL_{W^T}$. So we can do the same geometric construction for $(W^T,G^T)$ as we did for $(W,G)$, producing a K3 surface $X_{W^T,G^T}$ and we arrive at the following definition of BHK mirror symmetry for K3 surfaces.
\begin{definition}
 The {\em BHK mirror} (dual manifold) of the K3 surface $X_{W,G}$, with $J_W\leq G\leq SL_W$, is $X_{W^T,G^T}$.\footnote{In fact this definition can be extended to CY orbifolds using the same construction, and to their crepant resolutions, if they exist. But we want to avoid these technicalities here.}
\end{definition}

\section{Main Result and New Methods}\label{s:MainResult}

We have introduced two forms of mirror symmetry for K3 surfaces, namely BHK mirror symmetry and LPK3 mirror symmetry. The question is, if both versions of mirror symmetry apply to a given K3 surface, do they agree? 

We turn our attention specifically to K3 surfaces $X_{W,G}$ defined by the pair $(W,G)$ with $W$ an invertible polynomial of the form 
\begin{equation}\label{eq:poly}
W=x_0^n+g(x_1,x_2,x_3),\quad n\geq 2.
\end{equation}
that is quasihomogeneous with respect to one of the 95 weight systems of Reid and Yonemura and $G$ satisfying $J_W\leq G\leq \SL_W$. As previously mentioned, 
if $W$ satisfies these two conditions, then $W^T$ does as well. In what follows, to ease notation, we fix a pair $(W,G)$ and denote $X=X_{W,G}$, and $X^T=X_{W^T,G^T}$, whenever there is no possibility of confusion.

The natural question which arises is the following: are $X$ and $X^T$ LPK3 mirror? In other words, can we find an integral lattice $M$, which embeds primitively into $\Pic(X)$, such that $M^\vee$ also embeds primitively into $\Pic(X^T)$? This is what we mean for the BHK and LPK3 mirror symmetry constructions ``to agree". 

The answer to this question is given in the affirmative in the following theorem. 

\begin{theorem}\label{t:main}
Let $W$ be an invertible polynomial of the form $x_0^n+g(x_1,x_2,x_3)$ quasihomogeneous with respect to one of Reid and Yonemura's 95 weight systems, and $G$ a group of symmetries of $W$ satisfying $J_W\leq G \leq SL_W$. Then the K3 surface $X_{W,G}$ and its BHK mirror $X_{W^T,G^T}$ form an LPK3 mirror pair.
\end{theorem}

Theorem~\ref{t:main} has already been proven for $n=2$ by Artebani-Boissière-Sarti \cite{ABS}, for the case when $n$ is prime (and not equal to 2) by Comparin-Lyon-Priddis-Suggs \cite{CLPS}, and for $n$ composite, except for $n=4,8,12$ by Comparin-Priddis \cite{CP}. In the present work we complete the proof by looking at the remaining cases $n=4,8,12$. 

In all cases, Theorem~\ref{t:main} is proven by finding an appropriate lattice polarization of $X$, determined by a specific non-symplectic automorphism of $X$ of order $n$. Indeed, we observe that when $W$ is as in \eqref{eq:poly}, the surface $X$ admits the purely non-symplectic automorphism of order $n$ defined by 
\begin{equation}\label{e:sigman}
\sigma_n(x_0,x_1,x_2,x_3)=(\zeta_nx_0,x_1,x_2,x_3).    
\end{equation}
Here $\zeta_n$ denotes a primitive $n$-th root of unity. The invariant lattice $S(\sigma_n)$ polarizes $X$, i.e. the inclusion map $S(\sigma_n)\hookrightarrow\Pic(X)$ is a primitive embedding. 

The defining polynomial $W^T$ of the BHK mirror K3 surface $X^T$ also has form \eqref{eq:poly}, and so it admits a non-symplectic automorphism $\sigma_n^T$ of the same form \eqref{e:sigman}, and so the invariant lattice $S(\sigma_n^T)$ polarizes $X^T$. This gives us two candidates for $M$ and $M^\vee$, namely $S(\sigma_n)$ and $S(\sigma_n^T)$. 

The proof of Theorem~\ref{t:main} consists of verifying that $S(\sigma_n)^\vee\cong S(\sigma_n^T)$, i.e. $S(\sigma_n)$ and $S(\sigma_n^T)$ are mirror lattices in the sense of Definition~\ref{def:LPK3}. 
By Lemma~\ref{lem:orthogonal} and Proposition~\ref{prop:existenceuniqueness}, we can verify this fact by finding the rank and quadratic form of the invariant lattices $S(\sigma_n)$ and $S(\sigma_n^T)$. Indeed if we let $r_X=\rk S(\sigma_n)$, $r_{X^T}=\rk S(\sigma_n^T)$ and let $q_X$ and $q_{X^T}$ denote the respective discriminant quadratic forms, we know $S(\sigma_n)$ and $S(\sigma_n^T)$ are mirror lattices if 
\[
r_X=20-r_{X^T} \qquad \text{ and }\qquad q_X=-q_{X^T}.
\]

For $n$ prime, $r_X$ and $q_X$ can be computed only from the topological data of the fixed locus (see \cite{ABS}, \cite{CLPS}). 
Chiodo--Kalashnikov--Veniani \cite{ChKV} have an alternate proof that mirror symmetry holds in case $n$ is prime, which uses what they call semi--Calabi--Yau varieties. Both proofs rely on the fact that $S(\sigma_n)$ has a certain form (called $p$-elementary) when $n$ is prime (see \cite{CP} for a discussion of the differences in these methods). 

However, when $n$ is composite, neither of these methods determine the lattice uniquely. In this case, the authors of \cite{CP} determined $S(\sigma_n)$ by actually finding generators of the lattice, which are certain special divisors on $X$. 
Once the lattice is known, one can compute $r_X$ and $q_X$ directly to verify LPK3 mirror symmetry. 
The main difficulty with $n=4,8,12$ is showing that the proposed set of generators actually generate $S(\sigma_n)$. We will describe those generators now. 

Recall $Z_{W,G}$ is defined as the quotient $Z_W/\widetilde{G}$; the generators of the invariant lattice come from the resolution of singularities $\pi:X\to Z_{W,G}$. There are two important classes of curves on $X$. First of these are what we call the ``coordinate curves''. We denote by $C_{x_i}$ the (strict transforms of the) divisor defined by ${x_i=0}$. Let $\cK$ the set of smooth irreducible components of the coordinate curves $C_{x_i}$ for $i=0,1,2,3$ (if any of the irreducible components are singular, we omit them). Clearly, $\sigma_n$ leaves each of the divisors $C_{x_i}$ invariant (but not necessarily pointwisely), so $\sigma_n$ acts on $\cK$. For $C\in \cK$, we define $G_C$ as the isotropy group, $\cK/\sigma_n$ as the set of orbits, and 
\[
b_C=\tfrac 1{|G_C|}\sum_{i=0}^{n-1}\sigma_n^iC.
\]
Notice that if $C\in\cK$ is smooth, then $b_C=C$. 
In general, $b_C$ is a sum of some irreducible components of one of the $C_{x_i}$.  

Next, we look at the exceptional curves of $\pi:X\to Z_{W,G}$ and define $\cE$ as the set of exceptional curves of $\pi$. Because $W$ is nondegenerate, all of the singularities of $Z_{W,G}$ lie on the coordinate curves, and so $\sigma_n$ acts on the set $\cE$. Given $E\in\cE$, let $G_E$ denote the isotropy group of $E$ and let $\cE/\sigma_n$ denote the set of orbits of this action. As with the coordinate curves, we set 
\[
b_E=\tfrac 1{|G_E|}\sum_{i=0}^{n-1}\sigma_n^iE.
\]
This is the sum of all curves in the orbit of $E$. In particular, if $E$ is invariant under $\sigma_n$, then $b_E=E$. 

The first step in determining $S(\sigma_n)$ is to determine the rank $r_X$. This is accomplished by the following lemma:

\begin{lemma}[{\cite[Lemma 3.5]{CP}}]\label{l:rank}
The rank $r_X$ of $S(\sigma)$ is equal to 1 plus the number of orbits of exceptional curves in the blow-up $X_{W,G}\to Z_{W,G}$, i.e. 
\[
\rk S(\sigma_n)=1+|\cE/\sigma_n|.
\]
\end{lemma}

The second step in determining $S(\sigma_n)$ is to find a minimal set of generators. 
It turns out the generators always come from the coordinate curves and the exceptional curves. Let $\cB=\set{b_E\mid E\in \cE/\sigma_n}\cup \set{b_C\mid C\in \cK/\sigma_n}$, and let $L_\cB$ denote the lattice generated by $\cB$. By construction, $L_\cB$ is clearly a sublattice of $S(\sigma_n)$, and in general, $\cB$ will have more than $r_X$ elements, so some of these generators are redundant and we can omit them. In any case, since we have an explicit set of generators, one can easily compute that $\rk L_\cB=r_X$, hence $S(\sigma_n)$ is an overlattice of $L_\cB$. We can then compute the discriminant quadratic form of $L_\cB$.  

The final step is to show that $L_\cB=S(\sigma_n)$. With $n$ composite, this can often be done by considering the possible overlattices of $L_\cB$, which are characterised by isotropic groups of the discriminant group $A_{L_\cB}$ (see \cite[Prop. 1.4.1]{nikulin}). In most cases, there are no such isotropic subgroups, and therefore no proper overlattices. Hence in these cases, one knows that $L_\cB=S(\sigma_n)$ (see also \cite[Method I]{CP}). 

The trouble comes when the lattice $L_\cB$ has many overlattices, as is almost universally the case when $n=4,8,12$. In this case, we can show $L_\cB=S(\sigma_n)$ by showing that the inclusion $L_\cB\hookrightarrow S(\sigma_n)$ is a primitive embedding. If this is true, then since they both have the same rank, we get equality, $L_\cB=S(\sigma_n)$.  

The main method for showing $L_\cB\hookrightarrow S(\sigma_n)$ is a primitive embedding is to find some other lattice $M$ together with two primitive embeddings $\iota_L:L_\cB\hookrightarrow M$ and $\iota_\sigma:S(\sigma_n)\hookrightarrow M$, such that the inclusion $\iota:L_\cB\hookrightarrow S(\sigma_n)$ satisfies $\iota_L=\iota_\sigma\circ\iota$ as in the following lemma.

\begin{lemma}\label{lem:embed}
Let $L$, $K$ and $M$ be integral lattices with injective maps $\iota_1:L\to M$, $\iota_2:K\hookrightarrow M$ and $\iota_3:L\hookrightarrow K$ such that $\iota_1=\iota_2\circ\iota_3$ as in the following diagram: 
\begin{center}
\begin{tikzpicture}
  \matrix (m)
    [ matrix of math nodes,
      row sep    = 3em,
      column sep = 4em
    ]
    {
      L              &  M \\
      K &             \\
    };
  \path
    (m-1-1) edge [{Hooks[right,length=0.8ex]}->] node [left] {$\iota_3$} (m-2-1)
    (m-1-1.east |- m-1-2)
      edge [{Hooks[right,length=0.8ex]}->] node [above] {$\iota_1$} (m-1-2)
    (m-2-1.east) edge [{Hooks[right,length=0.8ex]}->] node [below] {$\iota_2$} (m-1-2);
\end{tikzpicture}
\end{center}

\begin{enumerate}
    \item If $\iota_3$ and $\iota_2$ are primitive embeddings, then $\iota_1$ is as well. 
    \item If $\iota_1$ is a primitive embedding, then $\iota_3$ is as well.
\end{enumerate}
\end{lemma}

\begin{proof}
This lemma follows essentially from the isomorphism theorems of abelian groups. In this proof, since each map is an embedding, we will identify each lattice with its image under the embedding for the sake of simplicity. Under the assumptions of the first statement, $K/L$ is a free abelian group, and $M/K$ is a free abelian group. Therefore, since $M/K\cong (M/L)\big /(K/L)$, we see that $M/L$ must also be free. For the second part, notice if $\iota_1$ is a primitive embedding, then $M/L$ is a free abelian group, and $K/L$ is a subgroup, and so also free. Hence, $\iota_3$ is primitive. 
\end{proof}

As was mentioned prior to Lemma~\ref{lem:embed}, the main tool for showing that $L_\cB= S(\sigma_n)$ is to find a primitive embedding from $L_\cB$ into some other lattice, which we know is primitively embedded into $\Pic(X)$. From the first part of the lemma this ensures that $L_\cB$ is primitively embedded into $\Pic(X)$. Since we know $S(\sigma_n)$ is primitively embedded into $\Pic(X)$, the second part of the lemma then ensures that $L_\cB$ is primitively embedded into $S(\sigma_n)$, and therefore they must be equal,i.e. $L_\cB=S(\sigma_n)$.  

For the three cases $n=4,8,12$, the challenge is to find a primitive embedding into some lattice that we know is primitively embedded into the Picard lattice for each possible choice of $(W,G)$. In order to simplify the task, 
we reduce the number of pairs $(W,G)$ for which we need to to this. We describe this reduction in the next sections.  

\subsection{Isomorphisms of K3 surfaces} 

Our first reduction 
is to make equivalence classes of the surfaces $X_{W,G}$ based on certain isomorphisms, which preserve $S(\sigma_n)$. 
To do this, we use a theorem proven independently by Kelly \cite{kelly} and Shoemaker \cite{shoemaker} to find isomorphism classes within our desired list of K3 surfaces and to show that these isomorphisms preserve the non-symplectic automorphism (i.e. the explicit form of $\sigma_n$). 

\begin{theorem}[Kelly \cite{kelly}, Shoemaker \cite{shoemaker}]\label{t:kelly}
If $W$ and $W'$ are invertible polynomials and $J_{W^T} \leq G^T \leq SL_{W^T}$, $J_{(W')^T} \leq (G')^T \leq SL_{(W')^T}$ with $G^T=(G')^T$, then $Z_{W,G}$ and $Z_{W',G'}$ are birationally equivalent. 
\end{theorem}

In other words, if $(W,G)$ and $(W',G')$ are two different pairs (even from different weight systems), but they are both dual under BHK mirror construction to pairs satisfying $G^T=(G')^T$, then the varieties $Z_{W,G}$ and $Z_{W',G'}$ are birational, and so by extension, the K3 surfaces $X_{W,G}$ and $X_{W',G'}$ are birational. 

Kelly proved Theorem~\ref{t:kelly} using Shioda maps, and provided an explicit birational map. We go into more specifics about his arguments in our proof of Lemma \ref{lem:equiv} below.

Let $X$ and $X'$ be two of the K3 surfaces with non-symplectic automorphism we have introduced, and $\phi:X\to X'$ be Kelly's birational map. Since our K3 surfaces have the additional structure of a purely non-symplectic automorphism, we need to verify that the birational map of Kelly preserves the automorphism $\sigma_n$. In this discussion, we will denote the automorphism of order $n$ on $X$ and the automorphism of order $n$ on $X'$ both by $\sigma_n$. 
The condition we want then is for all $x\in X$, we have $\phi(\sigma_n \cdot x)=\sigma_n\cdot\phi(x)$. A map that preserves $\sigma_n$ in this way will be called {\em $\sigma_n$-equivariant}.

\begin{lemma}\label{lem:equiv}
The birational equivalence given in Theorem \ref{t:kelly} is $\sigma_n$-equivariant.
\end{lemma}

\begin{proof}
The proof of Theorem \ref{t:kelly} uses what are called Shioda maps, which are rational maps. We introduce the following notation. Let W be an invertible polynomial with exponent matrix $A_W$, and let $B = dA_W^{-1}$, where $d$ is any integer such that $B$ has integer entries. Further, let $X_{dI}$ be the hypersurface in $\PP^n$ cut out by the Fermat polynomial $x_0^d + ... + x_n^d$. Then, the Shioda map $\phi_B: Z_{dI} \to Z_W$ is given by $\phi_B(y_0:...:y_n) = (x_0:...:x_n)$, where $x_j = \prod_{k=0}^n y_k^{b_{jk}}$. We denote by $\phi_{B^T}$ the Shioda map for $W^T$. Kelly proves that these Shioda maps push forward any action by a symmetry group $G$ via $\phi_{B_*}(g) = Bg$, and similarly for $W^T$. 

Now, Kelly's proof shows that if $G=G'$, then both $Z_{W^T,G^T}$ and $Z_{W'^T,G'^T}$ are birational to $Z_{dd'I}/H$, for some group $H$. 
Consider the automorphism $\sigma_{dd'}$ of $Z_{dd'I}$ given by 
\[
\sigma_{dd'}(x_0:\dots:x_n)=(e^{2\pi i/(dd')}x_0:x_1:\dots:x_n),
\]
which descends to an action on $Z_{dd'I}/H$. Notice that $\sigma_{dd'}$ has order $n$ in $Z_{dd'I}/H$, and this translates to the ation of $\sigma_n$ on both $Z_{W,G}$ and $Z_{W^T,G^T}$. \end{proof}

Thus in the situtation of Theorem~\ref{t:kelly}, we also obtain a birational map $X\to X'$ which is $\sigma_n$-equivariant. 

The following theorem is a standard result in the minimal model program in algebraic geometry and the classification of compact complex surfaces in algebraic geometry. A proof can be found in \cite{surfaces}. 

\begin{theorem}\label{t:iso}
Every birational map between K3 surfaces is an isomorphism.
\end{theorem} 

To summarize, if we have two pairs $(W,G)$ and $(W',G')$ such that $W$ and $W'$ have the same weight system, and $G^T=(G')^T$, then the corresponding K3 surfaces $X$ and $X'$ are isomoprhic, and the isomorphism preserves the action of $\sigma_n$. Thus we have reduced to finding a method for computing $r_X$ and $q_X$ for just one representative in each isomorphism class. We will further reduce the problem using certain deformations of our K3 surfaces.

\subsection{Deformations of K3 surfaces}\label{ss:deformations}
In this section we describe a second reduction involving deformations that preserve the invariant lattice. 
We begin with two pairs $(W,G)$ and $(W',G')$ such that $W=x_0^n+f(x_1,x_2,x_3)$ and $W'=x_0^n+f'(x_1,x_2,x_3)$ with $W$ and $W'$ invertible polynomials with respect to the same weight system system and with $G=G'$. We can then form a family of hypersurfaces $\cZ\subseteq W\PP^n(q_0,\ldots,q_3)\times\AA^1$ defined by 
\[
x_0^n+tf(x_1,x_2,x_3)+(1-t)f'(x_1,x_2,x_3)
\] 
for $t \in \AA^1$. On an open subset $U\subseteq\AA^1$, these hypersurfaces will be quasismooth (see \cite{CP} for more discussion on this condition). Furthermore, $\tilde G$ acts on this family, by acting in each fiber. We can then form the quotient $\cZ/\tilde G$. On $U$, the singularities of each fiber will be in the coordinate curves, giving us sections $U\to \cZ$, which we can blow up to resolve the singularities in each fiber. 

This yields a family $\cX\to U$ of K3 surfaces, such that the fiber $X_0$ over $0$ is $X_{W,G}$, and the fiber $X_1$ over 1 is the surface $X_{W',G'}$. Even more, in each fiber of this family, we can we can define the set $\cB'_t$ for $t\in U$ generated by the set $\cE$ of exceptional curves and the set $\cK$ of smooth coordinate curves of the blowup, as in Section~\ref{s:MainResult}.\footnote{We use $\cB'=\cE\cup \cB$ instead of $\cB$ here for simplicity. Clearly $L_\cB\subset L_{\cB'}$ is a primitive embedding. Lemma~\ref{lem:embed} will again yield the desired primitive embedding for $L_\cB$.}
Let $L_{\cB_t'}$ be the lattice generated by the curves over $t\in U$.

We aim to show that if $L_{\cB'_0}\subseteq\Pic(X_0)$ is primitively embedded for the fiber $X_0$, then the same is true for $L_{\cB'_1}\subseteq\Pic(X_1)$ in the fiber $X_1=X_{W',G'}$. 

Let us denote the general fiber of the family $\cX\to U$ by $\eta$, and the geometric fiber corresponding to $\eta$ by $X_{\bar\eta}$. By Proposition~2.10 and Remark~2.13 in \cite{huybrechts}, we get a specialization map 
\[
\spec:\Pic(X_{\bar \eta})\to \Pic(X_0)
\]
which is a primitive embedding. The same is true replacing $X_0$ by $X_1$. 

By construction, we have the following commutative diagram:
\begin{center}
\begin{tikzpicture}
  \matrix (m)
    [ matrix of math nodes,
      row sep    = 3em,
      column sep = 4em
    ]
    {
       L_{\cB_0}   &  \Pic(X_0) \\
     L_{\cB_{\bar\eta}}  &  \Pic(X_{\bar\eta})       \\
    };
  \path
    (m-1-1) edge [->] node [left] {$\cong$} (m-2-1)
     (m-2-2) edge [{Hooks[right,length=0.8ex]}->] node [right] {$\spec$}(m-1-2)
    (m-1-1.east |- m-1-2)
      edge [{Hooks[right,length=0.8ex]}->] (m-1-2)
    (m-2-1.east) edge [{Hooks[right,length=0.8ex]}->] (m-2-2);
\end{tikzpicture}
\end{center}
and a similar diagram replacing $X_0$ by $X_1$. 

By the second part of Lemma~\ref{lem:embed}, assuming we know that $L_{\cB_0}\hookrightarrow \Pic(X_0)$ is a primitive embedding, we have $L_{\cB_\eta}\hookrightarrow\Pic(X_{\bar\eta})$ is a primitive embedding. Considering the similar diagram for $X_1$, since $\spec:\Pic(X_{\bar\eta})\to\Pic(X_1)$ is a primitive embedding, and $L_{\cB_\eta}\hookrightarrow\Pic(X_{\bar\eta})$ is a primitive embedding, we have that  $L_{\cB_1}\hookrightarrow\Pic(X_{1})$ is a primitive embedding. 

In this way, we again reduce the number of pairs $(W,G)$ for which we must check $L_\cB=S(\sigma_n)$, since any two pairs $(W,G)$ and $(W',G')$ satisfying the conditions outlined above will yield two K3 surfaces $X$ and $X'$ with the same invariant lattice, as soon as we know that $L_\cB$ is primitively embedded into the Picard lattice for one of them.
Let us illustrate this process with the following example. 

\begin{example}

Suppose we look at the two polynomials
\[
W=x^4+y^4+z^3w+zw^3, \quad W'=x^4+y^4+z^3w+w^4
\]
and define the family of quartics in $\PP^3$ defined by the equation
\[
x^4+y^4+z^3w+tzw^3+(1-t)w^4.
\]     
Here $x,y,z,w$ are coordinates on $\PP^3$ and $t$ is the coordinate on the base $\AA^1$. 
We want the quotient $\cZ$ of these curves by the group $\tilde G\cong\ZZ/2\ZZ$, where the action is 
\[
(x,y,z,w,t)\to (-x,-y,z,w,t).
\]
Each fiber in the quotient has singularities at the points defined by $x=y=0$ and points defined by $z=w=0$. There are eight such points in each fiber, giving us eight sections $U\to \cZ$.
We blow them up to obtain a family $\cX\to U$, and a configuration of curves in each fiber that looks like in Figure \ref{fig:resIII}.

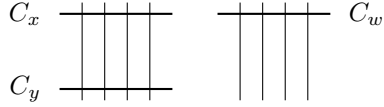
\begin{figure}[ht]
	\centering\begin{tikzpicture}[xscale=.6,yscale=.5]
	
	\draw [thick](-1.5,3)--(1,3);
	\node [left] at (-1.7,3){$C_x$};
	\draw [thick] (-1.5,1)--(1,1);
	\node [left] at (-1.7,1){$C_y$};
	\draw [thick] (2,3)--(4.5,3);
	\node [right] at (4.7,3){$C_w$};
	\draw (2.5,3.3)--(2.5,.7);
	\draw (3,3.3)--(3,.7);
	\draw (3.5,3.3)--(3.5,.7);
	\draw (4,3.3)--(4,.7);
	
	\draw (-1,3.3)--(-1,.7);
	\draw (-.5,3.3)--(-.5,.7);
	\draw (0,3.3)--(0,.7);
	\draw (.5,3.3)--(.5,.7);
	\end{tikzpicture}
	\caption{Resolution of singularities for each fiber}
	\label{fig:resIII}
\end{figure}

Here $C_x$, $C_y$ and $C_w$ are the strict transforms of the curves defined by $x=0$, $y=0$ and $z=0$, in each fiber, resp. On each fiber over $U$, the curves $C_x$ and $C_y$ are smooth of genus one, but the curve $C_w$ consists of four rational curves meeting at a point. We have not depicted the interstection points between these coordinate curves. In particular, $C_w$ intersects one of the exceptional curves connecting $C_x$ and $C_y$. 

Every fiber of this family has the configuration of curves of Figure \ref{fig:resIII}. Furthermore, the general fiber is a K3 surface with this configuration of curves. Let $L_{\cB'_t}$ be the lattice generated by this configuration in the fiber $X_t$. 

Suppose that we know $L_{\cB_0}$ is primitively embedded into $\Pic(X_0)$ in the fiber over $t=0$ (see the next section for verification of this fact). By the discussion above this is true also for the configuration of curves in the fiber over $t=1$.
\end{example}

In summary, using isomorphisms and deformations, we can group the K3 surfaces $X_{W,G}$ into eqivalence classes, where we know that each K3 surface in each equivalence class has the same invariant lattice. Thus in order to prove Theorem~\ref{t:main}, we need only compute $S(\sigma_n)$ for one representative $X$ in each class, and then compare the invariants with the equivalence class containing $X^T$.

\section{Proof of main theorem}\label{s:proofMain}

In the preceding section we have described the methods for grouping K3 surfaces $X_{W,G}$  into equivalence classes based on isomorphism and deformation and we have shown that every K3 surface in a given equivalence class has the same invariant lattice $S(\sigma_n)$. 
The proof of Theorem~\ref{t:main} now relies on our being able to verify that $L_\cB=S(\sigma_n)$ in each case. 
We summarize the process and give more  specific details in Section \ref{computations}, outlining the computational part of the proof.

The process is as follows: 
\begin{enumerate}
    \item For $n=4,8,12$, we list all pairs $(W,G)$, where $W=x_0^n+g(x_1,x_2,x_3)$ is invertible and $J_W\leq G\leq SL_W$. For some polynomials, there are several choices of $G$, each giving rise to a different K3 surface.  This is shown in Tables \ref{tab:order4}, \ref{tab:order8} and \ref{tab:order12} in Appendix \ref{app:tables}. 
    
    \item We group K3 surfaces according to the rank $r_X$ of their invariant lattice $S(\sigma_n)$. 
    These ranks are computed using Lemma~\ref{l:rank}. 
    
    \item In a given rank, we find isomorphism classes; in other words, we check whether $G^T=(G')^T$ for any two pairs $X_{W,G}$ and $X_{W'G'}$, and group these together. 
    
    \item In a given rank, we look for K3 surfaces that have the same weights and group, and find deformations, as described in Section~\ref{ss:deformations} and group these together. 
    
    Together with the previous step, this grouping defines isomorphism/deformation classes of K3 surfaces, such that each member in a given class has the same invariant lattice.   
    
    \item\label{step:rep} In each isomorphism/deformation class, choose an approprate representative and check that $L_\cB=S(\sigma_n)$ for that representative (recall the definition of $L_{\cB}$ in Section~\ref{s:MainResult}). We will indicate in the Tables \ref{tab:order4}, \ref{tab:order8}, and \ref{tab:order12}, which representative we have chosen.  
    
    \item Using the explicit description $S(\sigma_n)=L_\cB$, we can compute $r_X$ and $q_X$ for each invariant lattice $S(\sigma_n)$. 
    
    \item For each pair $X, X^T$ we can then check that $S(\sigma_n)^\vee=S(\sigma_n^T)$ by comparing the ranks and discriminant quadratic forms, as described in Section~\ref{s:MainResult}. 
    
\end{enumerate}

Once this last step has been verified, the proof of Theorem~\ref{t:main} is complete. 

There are a handful of cases where the lattice $L_\cB$ admits no proper overlattices. 
Since $S(\sigma_n)$ is an overlattice of $L_\cB$, we obtain immediately that $L_\cB=S(\sigma_n)$. In particular this is true for most of the cases with $n=12$. We will make note of these specific cases later.

In step \ref{step:rep} the majority of equivalence classes have a representative such that $\tilde G$ is the trivial group. If this is the case we have a primitive embedding $L_\cB\subseteq\Pic(X)$ from a result of Belcastro \cite{belcastro}, as we will describe in the Remark below. 
If, on the other hand, there is no representative with trivial group in a given equivalence class, we call that class {exceptional} and the computation of the primitive embedding will be case specific.
There are three particular cases that require more work and we discuss these in Section~\ref{sec:exc}.

\subsubsection{Primitive Embeddings:}
\label{ss:belcastro}

Once we have the K3 surfaces sorted into classes, it suffices to show that $L_\cB$ is primitively embedded into $\Pic(X)$ for only one of the representatives in each class. In most cases there is representative of each class defined by a pair of the form $(W,J_W)$; i.e. $G=J_W$. For each of the 95 weight systems, Belcastro \cite{belcastro} computed the Picard lattice for a general member of the family of K3 surfaces defined by nondegenerate quasihomogeneous polynomials with respect to that weight system. Each invertible polynomial $W$ with the same weight system yields a special member of this family, but often with bigger Picard lattice. We can use the result of Belcastro to show that the embedding $L_\cB\hookrightarrow\Pic(X)$ is primitive. 

More specifically, given a K3 surface $X=X_{W,J_W}$, we let $\cB'=\cK\cup\cE$ denote the set of exceptional curves and smooth irreducible components of the coordinate curves as in Section~\ref{s:MainResult}. We can construct a family $\cX\to B$ of K3 surfaces defined by the nondegenerate polynomials that are quasihomogeneous with respect to the same weight system as $W$, where $B$ is some open subset of some affine space. Let $X_{\bar\eta}$ denote the generic (geometric) fiber.

Each fiber $X_t$ over $t\in B$ is the minimal resolution of a hypersurface $Z_t$ in weighted projective space defined by one of the nondegenerate quasihomogeneous polynomials mentioned previously. One can check via Yonemura's calculations in \cite{yonemura} that the singularities of $Z_W$ are also singularities of $Z_t$. Thus we have an injective map $L_{\cB'}\hookrightarrow\Pic(X_{\bar\eta})$ fitting into the following commutative diagram.:

\begin{center}
\begin{tikzpicture}
  \matrix (m)
    [ matrix of math nodes,
      row sep    = 3em,
      column sep = 4em
    ]
    {
      L_{\cB'}              &  \Pic(X_{W,G}) \\
      \Pic(X_{\bar\eta}) &             \\
    };
  \path
    (m-1-1) edge [->] (m-2-1)
    (m-1-1.east |- m-1-2)
      edge [{Hooks[right,length=0.8ex]}->] (m-1-2)
    (m-2-1.east) edge [{Hooks[right,length=0.8ex]}->] node [below] {$\spec$} (m-1-2);
\end{tikzpicture}
\end{center}

Recall $\spec:\Pic(X_{\bar\eta)}\to\Pic(X)$ is a primitive embedding. Using the explicit description of $L_{\cB'}$ and Belcastro's computation of $\Pic(X_{\bar\eta})$ in \cite{belcastro}, one can check that the map $L_{\cB'}\hookrightarrow\Pic(X_{\bar\eta})$ is an isomorphism (see \cite{bott}). By Lemma~\ref{lem:embed}, the lattice $L_{\cB'}$ is a primitive sublattice of $\Pic(X)$. Furthermore, as before, since $L_\cB$ is a primitive sublattice of $L_{\cB'}$, we have that $L_\cB$ is a primitive sublattice of $\Pic(X)$. Let us emphasize that we require the group $G=J_W$. 

In \cite{belcastro}, Belcastro listed the Picard lattices using a number assigned to each weight system by Yonemura (see \cite{yonemura}). In Tables \ref{tab:order4}, \ref{tab:order8}, and \ref{tab:order12} in the Appendix, we have indicated this number in the last column. 

In order to explain the entire process, we give an ilustrative example for $n=4$ and invariant lattice of rank 4.

\begin{example}\label{ex:rk4}
For $n=4$ and $r_X=4$ 
there are three pairs $(W,G)$ satisfying the conditions we have outlined in this work. 
These are shown in Table~\ref{Order 4, Rank 4 Example}---a small excerpt from Table~\ref{tab:order4} in Appendix \ref{app:tables}. We have starred line 12 in the table, to indicate that we will use this pair to compute $S(\sigma_4)$ for the entire equivalence class. 

\begin{table}[h!]\centering
\begin{tabular}{ l c c c c c c c}
  No & Rk & Dual & Weights & Polynomial & $G/J_W$ & Form&Belcastro's no\\
  \hline 
  11 & 4 & 77 & (4,2,1,1;8) & $x^2+y^4+z^8+w^8$ & $\ZZ/2\ZZ$ & $w_{2,2}^{1}+w_{2,2}^5$&\\
  \hline
  12$^*$ & 4 & 78 & (3,2,2,1;8) & $x^2z+y^4+z^4+w^8$ & trivial & $w_{2,2}^{1}+w_{2,2}^5$ &19\\
  \hline
  13 & 4 & 79 & (3,2,2,1;8) & $x^2z+y^4+z^4+xw^5$ & trivial & $w_{2,2}^{1}+w_{2,2}^5$ &\\
  \hline
   \end{tabular}
 \caption{Order 4, Rank 4 Example} \label{Order 4, Rank 4 Example}
 \end{table}

Let $(W_{11}, G_{11})$ denote the pair 
for line 11 and, likewise, $(W_{12},G_{12})$ for line 12 and $(W_{13},G_{13})$ for line 13.
As $G_{12}=J_{W_{12}}$, we can use the method described in Section~\ref{ss:belcastro} to first compute the invariant lattice for line 12. We will then show that the three surfaces belong to the same isomorphism/deformation class, so that the invariant lattice in lines 11 and 13 are the same as in line 12.

In $\PP(3,2,2,1)$, the hypersurface $Z_{W_{12}}$ 
has four $A_1$ singularities at the points satisfying $x=w=0$ and an $A_2$ singularity at the point satisfying $y=w=0$. Blowing up these five points one obtains the configuration of curves depicted in Figure \ref{fig:rank4config}. We have labelled the strict transform of the set $\{y=0\}$ as $C_y$, and similary for $C_w$. The curve $C_y$ has genus 2, and $C_w$ has genus 0. 

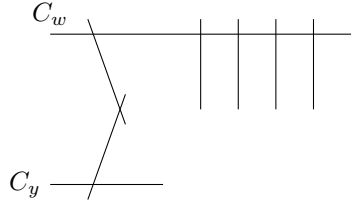
\begin{figure}[h!]
\centering
\begin{tikzpicture}

 \draw (0,0)--(1.5,0);
\node [above] at (0,2){$C_w$};
     
 \draw (0,2)--(4,2);
 \node [left] at (0,0){$C_y$};
 
 \draw (.5,-0.2)--(1,1.2);
 \draw (1,0.8)--(.5,2.2);
 
 \draw (2,2.2)--(2,1);
 \draw (2.5,2.2)--(2.5,1);
 \draw (3,2.2)--(3,1);
 \draw (3.5,2.2)--(3.5,1);
\end{tikzpicture}
\caption{Configuration of curves for  $W_{12}$}
\label{fig:rank4config}
\end{figure}

This configuration of curves has the incidence graph depicted in Figure \ref{fig:blow_up_12}. 
Let $R_1$ and $R_2$ denote the exceptional curves coming from the $A_2$ singularity and $R_3,\ldots, R_6$ the other four exceptional curves (coming from the $A_1$ singuarities). 

\begin{figure}[h]
\centering
\begin{tikzpicture}[xscale=.5,yscale=.4] 
 \begin{scope}[every node/.style={circle, draw, fill=black!50, inner sep=0pt, minimum width=4pt}]
   \node (n2) at (-2,0) {};
  \node (n3) at (-1,0) {};
  \node (n4) at (0,0) {};
  \node (n5) at (1,2) {};
  \node (n6) at (1,1) {};
  \node (n7) at (1,-1) {};
  \node (n8) at (1,-2) {};
\end{scope}

\node[{rectangle, draw, fill=black!50, inner sep=0pt, minimum width=4pt, minimum height=4pt}] (n1) at (-3,0) {};
\node[left] at (-3,0) {$C_y$};
\node[above] at (-0.3,0) {$C_w$}; 
 
  \foreach \from/\to in {n1/n2,n2/n3,n3/n4,n4/n5,n4/n6,n4/n7,n4/n8}
    \draw (\from) -- (\to);

\end{tikzpicture}
\caption{Graphs for the configuration of curves for $W_{12}$} \label{fig:blow_up_12}
\end{figure}
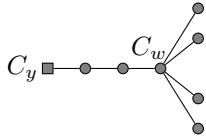

Define the set $\cB'=\{C_y, C_w,R_1,\dots, R_6\}$. If one considers the adjacency matrix for the graph in Figure \ref{fig:blow_up_12}, one can quickly see that $C_y$ is redundant, and that the curves in the configuration generate a lattice $L_{\cB'}$ with rank 7, and discriminant quadratic form $v\oplus w_{2,3}^1$, whose graph is shown in Figure~\ref{fig:Belcastro_ Pic_3221}. 

\begin{figure}[h]
\centering
\begin{tikzpicture}[xscale=.6,yscale=.5, thick, every node/.style={circle, draw, fill=black!50, inner sep=0pt, minimum width=4pt}]

  \node (n2) at (-2,0) {};
  \node (n3) at (-1,0) {};
  \node (n4) at (0,0) {};
  \node (n5) at (1,2) {};
  \node (n6) at (1,1) {};
  \node (n7) at (1,-1) {};
  \node (n8) at (1,-2) {};

  \foreach \from/\to in {n2/n3,n3/n4,n4/n5,n4/n6,n4/n7,n4/n8}
    \draw (\from) -- (\to);

\end{tikzpicture}
\caption{Generating set for $L_{\cB'}$.} \label{fig:Belcastro_ Pic_3221}
\end{figure}
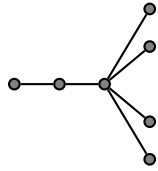

In Belcastro's computations \cite{belcastro}, the weight system $(3,2,2,1;8)$ corresponds to number 19, where we see that the Picard lattice for a generic K3 surface from this weight system has rank 7 and discriminant quadratic form $v\oplus w_{2,3}^1$. Thus $L_{\cB'}$ is equal to this lattice, and by Lemma~\ref{lem:embed}, we see that $L_{\cB'}$ is primitively embedding into $\Pic(X_{12})$.  
Now we consider the action of $\sigma_4$. This acts on $X_{12}$ by permuting $R_3, R_4, R_5, R_6$, and leaving each of $R_1$ and $R_2$ invariant. 
Thus the lattice $L_\cB$ is generated by $\cB=\{\mathcal C, R_1, R_2, R_3+R_4+R_5+R_6\}$, and clearly embeds primitively into $L_{\cB'}$, and hence embeds primitively into $\Pic(X_{12})$. Hence $S(\sigma_4)=L_\cB$. 

From this explicit description, we can again use the adjacency matrix, whose graph is depicted in Figure~\ref{fig:exampleInvariantGraph}, to compute that $L_\cB$ has rank 4, and discriminant quadratic form $\omega_{2,2}^1\oplus\omega_{2,2}^5$. 
This is the lattice $S(\sigma_4)=\langle 4\rangle\oplus A_3$. 
\begin{figure}[h]
\centering
\begin{tikzpicture}[xscale=.6,yscale=.5] 
 \begin{scope}[every node/.style={circle, draw, fill=black!50, inner sep=0pt, minimum width=4pt}]
  \node (n2) at (-2,0) {};
  \node (n3) at (-1,0) {};
  \node (n4) at (0,0) {};
  \end{scope}
  \node[{rectangle, draw, fill=black!50, inner sep=0pt, minimum width=4pt, minimum height=4pt}] (n5) at (1,0) {};
  \node[above] at (0.5,0) {4};

  \foreach \from/\to in {n2/n3,n3/n4,n4/n5}
    \draw (\from) -- (\to);

\end{tikzpicture}
\caption{Generating set for $L_{\cB}$. The square node represents the four permuted exceptional curves, and has self-intersection -8.} \label{fig:exampleInvariantGraph}
\end{figure}
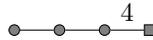

Now we have determined $S(\sigma_4)$ for $X_{12}$, we only need to show the other two representatives in this class have the same invariant lattice and this is done showing isomorphism/deformation. 
Observe that the dual polynomial for $W_{11}$ above is itself, while the dual polynomial for $W_{12}$ is $W_{12}^T=x^2+y^4+xz^4+w^8$. Both of these dual polynomials have the same weight system (4,2,1,1;8). 
From the definition of $G^T$, one can easily compute that the dual groups $G_{11}^T$ and $G_{12}^T$ for lines 11 and 12 are equal (e.g. as subgroups of $\GL_{4}(\CC)$). 
Theorem~\ref{t:iso}, $X_{11}$ and $X_{12}$ are isomorphic, and so their invariant lattices are the same.

Now observe that lines 12 and 13 have the same weight system and in both cases $G/J_W$ is the trivial group. 
We can create the a family of hypersurfaces in $\PP(3,2,2,1)$ via the polynomial $x^2z+y^4+z^4+tw^8+(1-t)xw^5$, with $t\in \AA^1$. On an open subest of $\AA^1$, each fiber of this family has the same singularities computed for $W_{12}$ and these five singularities in each fiber give us five sections of the family, which we can then blow up to obtain the same configuration of curves depicted in Figure~\ref{fig:rank4config} in each fiber, whose incidence graph is depicted in Figure~\ref{fig:blow_up_12}.
So by results above, lines 12 and 13 are deformations of one another that preserve the invariant lattice. 

In summary, lines 11 and 12 in Table~\ref{Order 4, Rank 4 Example} are connected via isomorphism, and lines 12 and 13 are connected by deformation, so all three have the same invariant lattice $S(\sigma_4)=\langle 4\rangle\oplus A_3$, with quadratic form $q=w^1_{2,2}+w^5_{2,2}$.
To emphasize the utility of these equivalence classes, notice that for line 11 we would not have been able to embed the lattice $L_\cB$ easily into some other lattice that we know.

From Table~\ref{tab:order4} in Appendix \ref{app:tables}, the BHK duals of lines 11, 12, and 13 are in lines 77, 78, and 79, resp. as shown in Table~\ref{Order 4, Rank 16 Example}. Again, we have starred line 79, since we use this example to compute $S(\sigma_n)$. 
\begin{table}[h!]\centering
\begin{tabular}{ l c c c c c c c}
  No & Rk & Dual & Weights & Polynomial & $G/J_W$ & Form & Belcastro's no\\
  \hline
 77 & 16 & 11 & (4,2,1,1;8) & $x^2+y^4+z^8+w^8$ & $\ZZ/4\ZZ$ & $w_{2,2}^{-1}+w_{2,2}^{-5}$& \\
 \hline
 78 & 16 & 12 & (4,2,1,1;8) & $x^2+y^4+xz^4+w^8$ & $\ZZ/4\ZZ$ & $w_{2,2}^{-1}+w_{2,2}^{-5}$ &\\
 \hline
 79* & 16 & 13 & (8,5,4,3;20) & $x^2z+y^4+z^5+xw^4$ & trivial & $w_{2,2}^{-1}+w_{2,2}^{-5}$&62 \\
 \hline 
   \end{tabular}
 \caption{Order 4, Rank 16 Example} \label{Order 4, Rank 16 Example}
 \end{table}
 
Via similar computations one sees that they belong to the same isomorphism/deformation class---lines 77 and 78 are isomorphic, and lines 78 and 79 are connected by deformation---and they have invariant lattice $S(\sigma_4)=U\oplus D_5\oplus D_9$ with form $q=w^{-1}_{2,2}+w^{-5}_{2,2}$.
This shows that the two mirror symmetry constructions agree for these K3 surfaces.

\end{example}

\subsection{Computations}\label{computations}
We list in tables in Appendix \ref{app:tables} all possibilities for an invertible polynomials $W$ of the form $W=x_0^n+g(x_1,x_2,x_3)$, $n=4$, $n=8$ and $n=12$---each case in a separate table---and the symmetry groups $G$, that generate the K3 surfaces $X_{W,G}$. 
In the table we make a change of variables from $x_0,\dots,x_3$ to $x,y,z,w$, to put the variables in order of descending weight. 

We have numbered the possible pairs $(W,G)$ to make referencing efficient. In each line of the table, we include the number of the pair, the rank of the invariant lattice, the number of its BHK dual according to our numbering, followed by the weight system, the polynomial, the group, and then the discriminant quadratic form of the invariant lattice. 

We have organized the numbering by rank, and have used double lines to separate the isomorphism/deformation classes in the table, the details of which are summarized afterward. Furthermore, we have indicated which representative in each class we have used to show that the embedding $L_\cB\subseteq \Pic(X_{W,G})$ is primitive with a $*$. 
The last column contains the number given to the weight system in \cite{belcastro} for comparison.

Finally, we have labelled with a superscript E those cases that are exceptional and we detail these computations in Section~\ref{sec:exc}. Some equivalence classes in the table have neither $*$ nor $E$; in these cases, the invariant lattice has no overlattices. This is also indicated in the last column.

We now describe how we obtain the isomorphism/deformation classes, using the methods described in Section~\ref{s:MainResult}. 
Theorem~\ref{t:main} can be verified by checking the rank and quadratic forms for each K3 surface and that of their BHK mirror.

\subsubsection{Order 4 Computations}\label{sec:order4}
Table~\ref{tab:order4} in Appendix \ref{app:tables} contains all possibilitites for an invertible polynomials $W$ of the form $W=x_0^4+g(x_1,x_2,x_3)$ together with the symmetry groups $G$.
We describe, according to $r_X$, the isomorphisms and deformations that group our K3 surfaces into equivalence classes. 
We reference each line in the table using the number we have assigned in the first column. 
\begin{itemize}
\item \textbf{Rank 1:} Lines 1--5 are all connected by deformations. The lattice $L_\cB$ for each of these K3 surfaces is $\langle 4\rangle$, which has no overlattices as described earlier (see also \cite[Method I]{CP}). So we obtain $S(\sigma_4)=\langle 4\rangle$ with discriminant quadratic form $w^1_{2,2}$. 

\item\textbf{Rank 2:} Lines 6--10 are related by deformation.

\item\textbf{Rank 4:} Example \ref{ex:rk4} shows that lines 11 and 12 are isomorphic and that lines 12 and 13 are related by deformation.

\item\textbf{Rank 5:} Lines 14--18 are related by deformation.

\item\textbf{Rank 6:} There are two different equivalence classes in this rank. 

Lines 19--21 are isomorphic and lines 21 and 22 are related via deformation and lines 19, 23, and 24 are related by deformation. Hence lines 19--24 all belong to the same class. 

Similarly, lines 25 and 26, 27 and 28, and 29 and 30 are pairwise isomorphic. Lines 25, 27, and 29 are related by deformation, thus lines 25--30 belong to the same class. 

\item\textbf{Rank 8:} Lines 31 and 32 are related by deformation. This is an exceptional case, and we will provide more details for line 31 in Section~\ref{ss:rank8}. 

\item\textbf{Rank 9:} Lines 33 and 34, and 35 and 36 are pairwise isomorphic. Furthermore, lines 34 and 36 are related by deformation. Thus lines 33--36 belong to the same class. 

\item\textbf{Rank 10:} For rank 10, there are 4 classes.

Line 37 is by itself, it is an exceptional case and will be studied in detail in Section \ref{ss:rank10}. 

Lines 38--40 are isomorphic.

Lines 41--43 are related by deformation. The K3 surface in line 42 has a non-symplectic automorphism $\sigma_2$ of order two generated by the set $\cB$ (see \cite{ABS}), hence $L_\cB=S(\sigma_2)=S(\sigma_4)$.

Lines 44 and 46--49 are isomorphic and 50--51 are isomorphic. Also, lines 45, 46, 52 and 53 are related by deformations. Thus lines 44--53 all belong to the same equivalence class. 

\item\textbf{Rank 11:} Lines 54 and 55, and 56 and 57 are pairwise isomorphic and lines 55 and 56 are related by deformation. So 54--57 belong to the same equivalence class.  

\item\textbf{Rank 12:} Lines 58 and 59 are isomorphic. As with line 42, the K3 surface in line 58 admits a non-symplectic automorphism $\sigma_2$ of order 2 whose invariant lattice $S(\sigma_2)$ is generated by the coordinate curves and exceptional curves (see \cite{ABS}). Hence $L_\cB$ embeds primitively into $S(\sigma_2)$, and $S(\sigma_2)$ embeds primitively into $\Pic(X)$. Thus by Lemma~\ref{lem:embed}, we know $L_\cB=S(\sigma_4)$.  

\item\textbf{Rank 14:} There are two equivalence classes with this rank. 

Lines 60--62 are isomorphic. Lines 60, 63 and 64 are related by deformations, as are 62 and 65. Thus 60--65 belong to the same class. This is an exceptional case, and we will work with the K3 surface in line 63 in Section \ref{ss:rank14}.

Lines 66--68 are isomorphic as are 69--71. Lines 66 and 69 are related by a deformation so that 66--71 belong to the same class. 

\item\textbf{Rank 15:} Lines 72--76 are isomorphic.

\item\textbf{Rank 16:} Lines 77 and 78 are related by deformation and lines 78 and 79 are isomorphic. Thus numbers 77--79 belong to the same class.

\item\textbf{Rank 18:} Lines 80--84 are isomorphic.

\item\textbf{Rank 19:} Lines 85--89 are isomorphic.
\end{itemize}

\subsubsection{Order 8 Computations}\label{sec:order8}

Table~\ref{tab:order8} in Appendix~\ref{app:tables} contains all possibilitites for an invertible polynomials $W$ of the form $W=x_0^8+g(x_1,x_2,x_3)$ together with the symmetry groups $G$.
Again in this section, organized by rank, we describe the isomorphisms and deformations that group our K3 surfaces into equivalence classes. 
As before, we reference each line in the table using the number we have assigned in the first column.

\begin{itemize}
\item\textbf{Rank 3:} Lines 1--4 and 6 are related by deformation and lines 4 and 5 are isomorphic. Thus lines 1--6 belong to the same equivalence class. 

\item\textbf{Rank 6:} Lines 7 and 8, 9 and 10, and 11 and 12 are pairwise isomorphic and lines 8, 10 and 12 are related by deformation. Thus lines 7--12 belong to the same equivalence class.

\item\textbf{Rank 7:} Lines 13 and 14 and lines 15 and 16 are pairwise isomorphic. Lines 14 and 16 are related by deformation. Thus lines 13--16 belong to the same equivalence class. 

\item\textbf{Rank 10:} Lines 17, 18 and 21 are related by deformation and lines 18--20 are isomorphic. Thus lines 17--21 belong to the same equivalence class. Moreover, the K3 surface in line 19 admits also a non-symplectic automorphism $\sigma_2$ of order 2 whose invariant lattice $S(\sigma_2)$ is generated by coordinate curves and exceptional curves. Hence $L_\cB$ is a primitive sublattice of $S(\sigma_2)$, and therefore a primitive sublattice of $\Pic(X)$ by Lemma~\ref{lem:embed}. Hence $L_\cB=S(\sigma_8)$.

\item\textbf{Rank 13:} Lines 22 and 23, and lines 24 and 25 are pairwise isomorphic and lines 23 and 25 are related by deformation. Thus numbers 22--25 belong to the same equivalence class. 

\item\textbf{Rank 14:} Lines 26--28, and lines 29--31 are isomorphic. Furthermore lines 27 and 29 are related by deformation. Thus lines 26--31 belong to the same equivalence class.  

\item\textbf{Rank 17:} Lines 32--36 are isomorphic and lines 32 and 37 are related by deformation. Thus lines 32--37 belong to the same equivalence class. 
\end{itemize}

\subsubsection{Order 12 Computations}\label{sec:order12}
Table~\ref{tab:order12} in Appendix \ref{app:tables} contains all possibilitites for an invertible polynomials $W$ of the form $W=x_0^{12}+g(x_1,x_2,x_3)$ together with the symmetry groups $G$.

Whenever the rank $r_X$ is equal to $2, 7, 8, 12, 13$, or $18$, the lattice $L_\cB$ has no overlattice. Thus in these cases, as mentioned before, we know that $L_\cB=S(\sigma_{12})$. Although unnecessary, one can also check that there is only one equivalence class for each of these ranks.

The only case remaining is when $r_X=10$. In Table~\ref{tab:order12}, lines 12, 13 and 15, and lines 16 and 17 are related by deformation, resp. and lines 14--16 are isomorphic. Thus lines 12--17 are all in the same equivalence class. 
This is an exceptional case. However, the K3 surface in line 13 admits a purely nonsymplectic automorphism $\sigma_4$ of order 4, which we have dealt with already (see line~45 in Table~\ref{tab:order4}). Furthermore, one can easily check that $S(\sigma_{12})=S(\sigma_4)=L_\cB$ as both lattices are generated by the coordinate curves and exceptional curves. 

Since there are no further details needed in any of these cases to show that $L_\cB=S(\sigma_{12})$, we have omitted the last column in Table~\ref{tab:order12}.

\section{Exceptional Calculations}\label{sec:exc}
According to what we have previously shown, the exceptional cases are the following ones:
\begin{enumerate}
    \item $n=4$, rank 14, lines 60--65. We will show invariant lattice for line 63 (see Section~\ref{ss:rank14}).
     \item $n=4$, rank 10, line 37. We will show its invariant lattice (see Section~\ref{ss:rank10}).
     \item $n=4$, rank 8, lines 31--32. We will show the invariant lattice for line 31 (see Section~\ref{ss:rank8}).
\end{enumerate}

\subsection{Rank 14}\label{ss:rank14}

We start with rank 14. For this rank, we use the pair $(W,G)$ with $W=x^4+y^4+z^4+w^4$ and the group $G/J_W\cong \ZZ_4$, which corresponds to line 63 in Table \ref{tab:order4}. There are actually six choices for such a group, but they are all the same up to permutation of the variables. A generator of $G/J_W$ is $\left(\tfrac{1}{4},\tfrac{3}{4},0,0\right )$.
After blowing up $A_{W,G}$ at the singular points, we get the configuration of curves in Figure \ref{fig:rank14exceptional}.
The curve $C_x$ and $C_y$ have genus zero and the other two have genus 1. Furthermore, $C_x$ intersects with each of $C_z$ and $C_w$ in a single point, not depicted here. The same is true for $C_y$. 

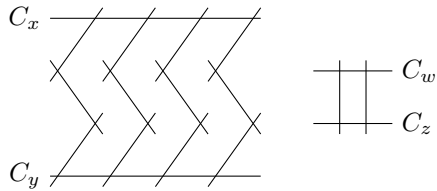
\begin{figure}
\begin{tikzpicture}[scale=0.7]
\draw (0,0) -- (4,0);
\node [left] at (0,0){$C_y$};
\draw (0,3) -- (4,3);
\node [left] at (0,3){$C_x$};
\draw (0,-0.2) -- (1,1.2);
\draw (1,0.8) -- (0,2.2);
\draw (0, 1.8) -- (1,3.2);

\draw (1,-0.2) -- (2,1.2);
\draw (2,0.8) -- (1,2.2);
\draw (1, 1.8) -- (2,3.2);

\draw (2,-0.2) -- (3,1.2);
\draw (3,0.8) -- (2,2.2);
\draw (2, 1.8) -- (3,3.2);

\draw (3,-0.2) -- (4,1.2);
\draw (4,0.8) -- (3,2.2);
\draw (3, 1.8) -- (4,3.2);
\draw (5,1) -- (6.5,1);
\node [right] at (6.5,1){$C_z$};
\draw (5,2) -- (6.5,2);
\node [right] at (6.5,2){$C_w$};
\draw (5.5,0.8) -- (5.5,2.2);
\draw (6,0.8) -- (6,2.2);
\end{tikzpicture}
\caption{Configuration of curves in the blow-up for $W=x^4+y^4+z^4+w^4$} \label{fig:rank14exceptional}
\end{figure}

The automorphism $\sigma_4=\left(\frac14,0,0,0\right)$ fixes the four $A_3$'s, but permutes the two $A_1$'s.
We call $\tau$ the involution $\tau=(\sigma_4)^2$. 
From Lemma~\ref{l:rank}, the invariant lattice $S(\tau)$ has rank 18 and is a 2-elementary lattice such that its discriminant group admits 4 generators.

Observe that the set $\cB$ consisting of coordinate curves and exceptional curves does not contain enough curves to generate a lattice with rank 18.  Instead, we need to add a few more curves to the generating set and we do this by considering the set of lines on the Fermat quartic in $\PP^3$. 
By \cite{SSvL}, \cite{degtyarev}, the Picard lattice of the Fermat quartic is generated by the 48 lines lying on the surface. 
Let $\omega=e^{2\pi i/8}$. 
For $0\leq j,k\leq 3$, following \cite{SSvL} let us define the following sets of lines, with 16 in each group, making up the 48 lines on the Fermat quartic:
\begin{align*}
\ell_1(j,k) &:=\set{[s,\omega i^js, t, \omega i^k t ]: [s,t]\in \PP^1}\\
\ell_2(j,k) &:=\set{[s,t,\omega i^js, \omega i^k t ]: [s,t]\in \PP^1}\\
\ell_3(j,k) &:=\set{[s, t, \omega i^k t ,\omega i^js]: [s,t]\in \PP^1}.
\end{align*}

In what follows, we only need the lines $\ell_1(j,k)$. Observe  that they are invariant under $\tau$, but they are permuted pairwise by $\sigma_4$.

Considering the group $G/J_W$, one has 
\begin{align*}
\left(\tfrac{1}{4},\tfrac{3}{4},0,0\right)\cdot [s,\omega i^js, t, \omega i^k t ] &= 
  [is,\omega i^{j+3}s, t, \omega i^k t ]\\
  &=[s',\omega i^{j+2}s', t, \omega i^k t ]
\end{align*}
where $s'$ is just a change of parametrization for $\PP^1$. 
Thus the action of $G/J_W$ has $8$ orbits among the set of lines $\ell_1(j,k)$ and we obtain 8 lines on the quotient $X_{W,G}$. 

We now need to know how these curves intersect in the configuration of curves given above. 
First observe that they intersect $x=0$ and $y=0$ transversally: we show this by considering the chart $w\neq 0$ in $\PP^1$. This leaves us with the equation $x^4+y^4+z^4+1=0$. The curve $x=0$ intersects the curve $y=0$ in four points. Each of these points also lies on exactly four of the lines $\ell_1(j,k)$, and no line contains two of these points. 

Let us first consider $\ell(1,1)$, and the rest of the intersections are very similar calculations. In this chart, the line $\ell_1(1,1)$ has the parametrization $(-i\omega^{-1}s,s, -i\omega^{-1})$, and this line intersects the curves $x=0$ and $y=0$ in the point $(0,0,\omega^{-1})$ (recall we are in the chart $w\neq 0$). Hence the tangent direction is $(-i\omega^{-1}, 1, 0)$. Now if we consider how the curve defined by $x=0$ on this surface (so in the $x=0$ plane we are considering the curve $y^4+z^4+1=0$), we see this has tangent direction $(0,4y^3,4z^3)$, which at the intersection point becomes $(0,0,4\omega^{-3})$. Hence these two curves are transversal. 

Without too much trouble, one can see that in the blowup (i.e. in the configuration of curves of Figure \ref{fig:rank14exceptional}) in fact the lines $\ell_1(j,k)$ intersect the middle curves on the $A_3$ singularities, and don't intersect any of coordinate curves $C_x,C_y,C_z,C_w$. Furthermore there are 8 lines, and each line connects exactly one pair of curves, one from the $A_3$ singularity and one from the $A_1$. 

So now we can write down a generating set for $S(\tau)$. Let us first label the curves in the configuration in Figure \ref{fig:rank14exceptional}, and then the lines $\ell_1(j,k)$ according to Table~\ref{tab1}. 
Lines are described by which of the exceptional curves they intersect. We do not specify which $j,k$ belong to each one. 
\begin{table}[h!]
    \centering
\begin{tabular}{|c|c||c|c|}
    \hline
	Curve number & description & Curve number & description \\
	\hline
	1 & $C_y$ & 19 & line connecting 12, 17 \\
	2-3-4 & $A_3$ & 20 & line connecting 12, 16\\
	5-6-7 & $A_3$ & 21 & line connecting 9, 16\\
	8-9-10 & $A_3$ & 22 & line connecting 9, 17 \\
	11-12-13 & $A_3$ & 23 &line connecting 6, 16\\
	14 & $C_x$ & 24 & line connecting 6, 17\\
	15 & $C_z$ & 25 & line connecting 3, 16\\
	16 & $A_1$ & 26 & line connecting 3, 17\\
	17 & $A_1$ &&	\\
	18 & $C_w$ &&	\\
	\hline
\end{tabular} 
\caption{Labeling curves in $S(\tau)$}
    \label{tab1}
\end{table}

Among these 26 curves, one possibility for the 18 generators of $S(\tau)$ are the curves whose label is  $3,4,5,6,7,8,9, 
10,11,12$, $13,14,15,16,17,19,21,23$, which translates to $C_x$, $C_z$, 11 of the twelve $A_3$ exceptional curves, both $A_1$'s, and 3 of the lines $\ell_1(j,k)$ (this is computed by MAGMA \cite{magma}, see also \cite[Appendix A]{CP}). 

In order to compute the lattice $S(\sigma_4)$ we show a set of curves giving a lattice of rank 14 and show a primitive embedding of it into $S(\tau)$.
Consider a new labeling of the curves as in Table~\ref{tab2}.

\begin{table}[h!]
    \centering
\begin{tabular}{|c|c|}
    \hline
	Curve number & description  \\
	\hline
	1 & $C_y$  \\
	2-3-4 & $A_3$ \\
	5-6-7 & $A_3$  \\
	8-9-10 & $A_3$ \\
	11-12-13 & $A_3$ \\
	14 & $C_x$ \\
	15 & $C_z$ \\
	16 & sum of the $A_1$'s \\
	17 & $C_w$	\\
	\hline
\end{tabular} 
\caption{Labeling curves in $S(\sigma_4)$}\label{tab2}
\end{table}

These curves are the elemenst of $\cB$ and therefore they generate the lattice $L_\cB$ with rank 14 contained in $S(\sigma_4)$ (which also has rank 14). One can check that a minimum set of generators is given by curves 3--16. Since those are also sums of the generators of $S(\tau)$, the embedding is primitive. 
Thus we can conclude that $S(\sigma_4)$ is the lattice generated by these 14 curves and it has quadratic form $v\oplus v_2$.

\subsection{Rank 10}\label{ss:rank10}

We now consider the same polynomial $W=x^4+y^4+z^4+w^4$ and the group $G/J\cong \ZZ_2\times \ZZ_2$. There is only one such group, and it is generated by $\left(\left(\tfrac{1}{2}, \tfrac{1}{2},0,0\right),\left(\tfrac{1}{2}, 0,\tfrac{1}{2},0\right)\right)$. 
The automorphism $\sigma_4$ is $\sigma_4:(x:y:z:w)\mapsto (ix:y:z:w)$.
After blowing up, we get the configuration of curves of Figure \ref{fig:rank10exc}.

\begin{figure}[h!]
\begin{tikzpicture}[scale=0.7]
	
\draw (0,0) -- (6,0);
\node [left] at (0,0){$C_y$};
	
\draw (0,3) -- (6,3);
\node [left] at (0,3){$C_x$};

\draw (1,1) -- (5,1);
\node [below] at (5,1){$C_w$};

\draw (2,2) -- (6,2);
\node [right] at (6,2){$C_z$};
	
\draw (0.7,-0.2) -- (0.7,3.2);
\draw (0.3,-0.2) -- (0.3,3.2);

\draw (1.7,0.8) -- (1.7,3.2);
\draw (1.3,0.8) -- (1.3,3.2);

\draw (5.7,-0.2) -- (5.7,2.2);
\draw (5.3,-0.2) -- (5.3,2.2);
	
\draw (2.7,-0.2) -- (2.7,1.2);
\draw (2.3,-0.2) -- (2.3,1.2);

\draw (3.7,0.8) -- (3.7,2.2);
\draw (3.3,0.8) -- (3.3,2.2);

\draw (4.7,1.8) -- (4.7,3.2);
\draw (4.3,1.8) -- (4.3,3.2);
\end{tikzpicture}
	\caption{Configuration of curves in the blow-up}\label{fig:rank10exc}
	\end{figure}
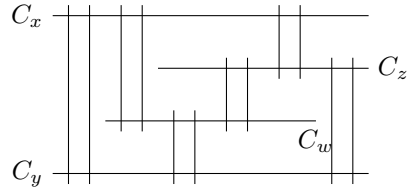

We consider again the lines $\ell_1(j,k), 0\leq j,k\leq 3$. We know that these curves intersect each other transversally and will meet the exceptional curves, because of the transversal intersection with the coordinate curves on the Fermat quartic as shown in the previous seciton. 
There are eight orbits of curves under the action of $G/J$ since:
\begin{align*}
\left(0,\tfrac{1}{2},0,\tfrac{1}{2}\right)\cdot [s,\omega i^js, t, \omega i^k t ] &= 
[s,\omega i^{j+2}s, t, \omega i^{k+2} t ]
\end{align*}
\begin{align*}
\left(0,\tfrac{1}{2},\tfrac{1}{2},0\right)\cdot [s,\omega i^js, t, \omega i^k t ] &= [s,-\omega i^js, -t, \omega i^k t ]\\
&=[s,\omega i^{j+2}s, t', \omega i^{k+2} t' ].
\end{align*}

Again, let $\tau=(\sigma_4)^2$. The action of $\tau$ on these lines is the following:
\begin{align*}
\left(\tfrac{1}{2},0,0,0\right)\cdot [s,\omega i^js, t, \omega i^k t ] &= 
[-s,\omega i^{j}s, t, \omega i^{k} t ]\\
&=[s',\omega i^{j+2}s', t, \omega i^k t ].
\end{align*}
This shows us that $\tau$ permutes curves $\ell_1(j,k)$ and $\ell_1(j+2,k)$, so that the sum $\ell_1(j,k)+\ell_1(j+2,k)$ is invariant for $\tau$. 

As before, we label curves from Figure \ref{fig:rank10exc} and lines; we detail it in Table~\ref{tab3}. The curves in the left side are exceptional curves, and the curves on the right side are the lines $\ell_1(j,k)$. 
\begin{table}[h!]
\begin{tabular}{|c|c||c|c|}
	\hline
	Curve number & description & Curve number & description \\
	\hline
	1 & $C_x$ & 17& $\ell(0,0)+\ell(0,2)$\\
	2 & $C_y$ & & intersects curve 5 and 15\\
	3 & $C_z$ &18& $\ell(0,1)+\ell(0,3)$\\
	4 & $C_w$ && intersects curve 5 and 16\\
	5-6 & curves intersecting $C_x$ and $C_y$ &19&$\ell(1,1)+\ell(1,3)$\\
7-8 & curves intersecting $C_x$ and $C_z$ &&intersecting curve 6 and 16\\
9-10 & curves intersecting $C_x$ and $C_w$ &20&$\ell(1,0)+\ell(3,0)$\\
11-12 & curves intersecting $C_y$ and $C_z$ &&intersecting curves 6 and 15\\
13-14 & curves intersecting $C_y$ and $C_w$ &&\\
15-16 & curves intersecting $C_z$ and $C_w$ &&\\\hline
\end{tabular} \caption{Labeling curves in $S(\tau)$}\label{tab3}
\end{table}

From Table \ref{tab3} and Figure~\ref{fig:rank10exc}, one can construct the matrix of incidence for the bilinear form generated by these curves. Notice the intersections in the right column will all have value 2 instead of 1, because the divisors listed there are sums of two genus zero curves. 

A computation with MAGMA \cite{magma} shows that these curves do indeed generate $S(\tau)$ and that the curves indexed by the following set are a minimal set of generators:
$1,2,3,6,8,9,10,11,12,13,14,15,16,17$. 
That corresponds to $C_x$, $C_y$, $C_z$, one of the curves between $C_x$ and $C_y$, one between $C_x$ and $C_z$, the lines $\ell_1(0,0)+\ell_1(2,0)$, and the rest of the exceptional curves.  

We know how $\sigma_4$ acts on the configuration of curves in Figure \ref{fig:rank10exc}. 
We consider the same curves 1-10 as before and 
curves 11--13 now represent the sums of the exceptional curves intersecting $C_y$ and $C_z$, $C_y$ and $C_w$ or $C_w$ and $C_z$, respectively, as in Table~\ref{tab4}. 
\begin{table}[h!]
\begin{tabular}{|c|c|}
\hline
	Curve number & description \\
	\hline
	1 & $C_x$\\
	2 & $C_y$ \\
	3 & $C_z$ \\
	4 & $C_w$ \\
	5-6 & curves intersecting $C_x$ and $C_y$ \\
7-8 & curves intersecting $C_x$ and $C_z$ \\
9-10 & curves intersecting $C_x$ and $C_w$ \\
11 & sum of curves intersecting $C_y$ and $C_z$ \\
12 & sum of curves intersecting $C_y$ and $C_w$ \\
13 & sum of curves intersecting $C_z$ and $C_w$ \\
\hline
\end{tabular}\caption{Labeling curves in $S(\sigma_4)$}\label{tab4}
\end{table}

We construct the incidence matrix of curves 1--13 and, by computations in MAGMA \cite{magma}, we can see that the lattice generated by the orbits of these curves has rank 10 and curves labeled by $1,2,3,6,8,9,10,11,12,13$ form a minimal set of generators for $S(\sigma_4)$. This is therefore a primitive embedding into $S(\tau)$ and so these in fact generate $S(\sigma_4)$.

Thus the invariant lattice $S(\sigma_4)=\langle 4\rangle\oplus(A_3)^3$ and its quadratic form is $w_{2,2}^1\oplus (w_{2,2}^5)^3$.

\subsection{Rank 8}\label{ss:rank8}

The last one is a little subtle, because we work in weighted projective space instead of projective space. 
For this rank, we use the pair $(W,G)$ with $W=x^2+y^4+z^8+w^8$ in $\PP(4,2,1,1)$ and the group $\tilde G\cong \ZZ_2$ generated by $(0,\tfrac{1}{2},\tfrac{1}{2},0)$. There are two other choices of group with $\tilde G\cong \ZZ_2$ in lines 11 and 42 of Table~\ref{tab:order4}, but these have invariant lattices with different rank. After blowing up we get the configuration of curves of Figure \ref{fig:rank8exc}. The curve $C_z$ and $C_w$ have genus zero, and $C_y$ has genus 1. The curve $C_x$ is not depicted, but has genus 3.

\begin{figure}
	\begin{tikzpicture}[scale=0.7]
	
	\draw (0,0) -- (5,0);
	\node [left] at (0,0){$C_w$};
	
	\draw (0,3) -- (5,3);
	\node [left] at (0,3){$C_z$};
	
	
	\draw (0,-0.2) -- (1,1.2);
	\draw (1,0.8) -- (0,2.2);
	\draw (0, 1.8) -- (1,3.2);
	
	\draw (1,-0.2) -- (2,1.2);
	\draw (2,0.8) -- (1,2.2);
	\draw (1, 1.8) -- (2,3.2);

	\draw (2.5,1.5) -- (5,1.5);
	\node [right] at (5,1.5){$C_y$};
	
	\draw (3,3.2) -- (3,1.3);
	\draw (3.5,3.2) -- (3.5,1.3);
	
	\draw (4,-0.2) -- (4,1.7);
	\draw (4.5,-0.2) -- (4.5,1.7);

	\end{tikzpicture}
	\caption{Configuration of curves in the blow-up}\label{fig:rank8exc}
	\end{figure}
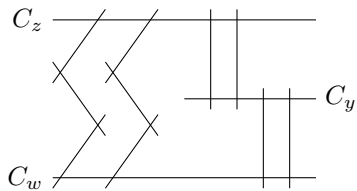

The automorphism $\sigma_4:(x:y:z:w)\mapsto(x:iy:z:w)$ fixes $C_y$ pointwisely, and therefore leaves all of the $A_1$'s invariant. It permutes the two $A_3$'s. 
It is not difficult to check that $\tau=(\sigma_4)^2$ fixes all of the depicted coordinate curves $C_y,C_z,C_w$, as well as the two central curves of the $A_3$'s. 
From the formula in Lemma~\ref{l:rank}, we learn that the invariant lattice $S(\tau)$ has rank 14, it is a 2-elementary lattice and its discriminant group has 6 generators. 

As before, the curves depicted in Figure~\ref{fig:rank8exc} do not generate a lattice of rank 14, so we need to add a few curves to generate the entire invariant lattice $S(\tau)$. 
Similar to what we have previously done, we describe some lines on the surface as follows. The subtlety here is that we can't embed $\PP^1$, but we can embed $\PP(2,1)$ into this weighted projective space. Let $\omega=e^{2\pi i/16}$ and define the lines 
\[
\ell(j,k) :=\set{[s^2,\omega^{2+4j} s, t, \omega^{1+2k} t ]: [s,t]\in \PP(2,1)}\\
\]
for $0\leq j\leq 1$ and $0\leq k\leq 7$. There are sixteen of these lines. Notice that this is a well-defined map into $\PP(4,2,1,1)$. When we blow up the surface, the singularity on each line at $z=w=0$ is resolved. 

Now we look to the group $G/J_W$ to see what these curves look like on the quotient. The group $G/J_W$ is generated by $\left(0,\tfrac{1}{2},0,\tfrac{1}{2}\right)$ and we see that 
\begin{align*}
\left(0,\tfrac{1}{2},0,\tfrac{1}{2}\right)\cdot [s^2,\omega^{2+4j}s, t, \omega^{1+2k} t ] &= 
[s^2,-\omega^{2+4j}s, t, -\omega^{1+2k} t ]\\
&=[s^2,\omega^{2+4(j+2)}s, t, \omega^{1+2(k+4)} t ]\\
&=[s^2,\omega^{2+4j}s, t, \omega^{1+2(k+4)} t ]
\end{align*}
The last equality holds because of the action on $\PP(2,1)$. Thus the action of $G/J_W$ has $8$ orbits among the set of lines $\ell(j,k)$ and we can just consider $\ell(j,k)$
for $0\leq j\leq 1$ and $0\leq k\leq 3$. 
Next we look at how these curves intersect each other and the curves in the configuration depicted in Figure \ref{fig:rank8exc}. 
As before, we can check that these lines intersect each other, and the coordinate curves transversally. 

Without too much trouble, one see that in the blowup (i.e. in the configuration of curves of Figure \ref{fig:rank8exc}) in fact the lines $\ell(j,k)$ intersect the middle curves on the $A_3$ singularities and the curve $C_y$ in the four points of intersection with $C_x$.  

Moreover, one can observe that $\ell(0,k)$ and $\ell(1,k)$ meet two different $A_3$'s and intersect with multiplicity 2.
So now we write down a generating set for $S(\tau)$. Let us first label the curves in Figure \ref{fig:rank8exc} and  the lines $\ell(j,k)$ according to Table \ref{tab5}. 

\begin{table}[h!]
    \centering
\begin{tabular}{|c|c|}
    \hline
	Curve number & description \\
	\hline
	1-2 & $A_1$'s intersecting $C_z$ and $C_y$  \\
	3 & $C_z$ \\
	4-5-6 & $A_3$ \\
	7 & $C_w$ \\
	8-9 & $A_1$'s intersecting $C_w$ and $C_y$ \\
	10-11-12 & $A_3$ \\
	13--16 & $\ell(0,k)$'s intersecting curve 11 \\
	17--20 & $\ell(1,k)$'s intersecting curve 5 \\
	21 & $C_y$ 	\\
	\hline
\end{tabular} \caption{Labeling curves in $S(\tau)$}\label{tab5}
\end{table}

Using MAGMA \cite{magma}, we can find that the lattice generated by these 21 curves has rank 14 and a minimal set of generators for $S(\tau)$ is formed by curves 2--15, which translates to $C_z$, $C_w$, 3 of the $A_1$'s, both $A_3$'s, and 2 of the lines $\ell_1(j,k)$, both intersecting curve number 5. 

Now we look for a primitive embedding of $S(\sigma_4)$ into $S(\tau)$. 
Consider curves as in Table~\ref{tab6}. Curves 4--6 each represents a sum of two curves permuted by $\sigma_4$ (recall that the two $A_3$'s are permuted by $\sigma_4$). 
\begin{table}[h!]
    \centering
\begin{tabular}{|c|c|}
\hline
	Curve number & description \\
	\hline
	1-2 & $A_1$'s intersecting $C_z$ and $C_y$  \\
	3 & $C_z$ \\
	4-5-6 & sum of the $A_3$'s \\
	7 & $C_w$ \\
	8-9 & $A_1$'s intersecting $C_w$ and $C_y$ \\
	\hline
\end{tabular} \caption{Labeling curves in $S(\sigma_4)$}\label{tab6}
\end{table}

These curves generate a lattice of rank 8 contained in $S(\sigma_4)$ (which also has rank 8). We can see from the calculation that in fact, a minimum set of generators is curves 1--8. Since those are also generators of $S(\tau)$, the embedding is primitive and we conclude that $S(\sigma_4)$ has quadratic form $v\oplus w_{2,2}^2\oplus w_{2,2}^5$.

\input{appendix1.tex}

\bibliographystyle{plain}
\bibliography{references}

\end{document}

%% file: appendix1.tex
\appendix
\section{Polynomials and forms}
\label{app:tables}
In tables below we list all possibilities for polynomials $W$ of the form $W=x_0^n+g(x_1,x_2,x_3)$, for $n=4,8,12$, and the group $G$, as explained in Section \ref{s:proofMain}. We have numbered each of the K3 surfaces, and have organized the numbering by rank. We have used double lines to separate the isomorphism/deformation classes in the table. Details for determining the equivalence classes are in Sections \ref{sec:order4}, \ref{sec:order8}, \ref{sec:order12}.  

We have indicated which representative in each class we have used to show that the embedding $L_\cB\hookrightarrow \Pic(X_{W,G})$ is primitive with a $*$. 
We have labelled the exceptional cases with a superscript $E$. 

In the last column of Tables~\ref{tab:order4} and \ref{tab:order8} we list one of three things. 
\begin{enumerate}
    \item If the lattice $L_\cB$ does not have any overlattices, as described previously, we write NOL. 
    \item If a line in the table is starred, the lattice computed by Belcastro in \cite{belcastro} has been used to show that $L_\cB$ primitively embeds into $\Pic(X)$ as described in Section~\ref{ss:belcastro}. In this case, we indicate the number of the weight system assigned in \cite{belcastro} for cross-referencing (see also \cite{yonemura}). 
    \item Finally, for three of the exceptional equivalence classes (indicated by a superscript $E$ on one of the representatives) we have provided further details regarding the primitive embedding of $L_\cB$ into $\Pic(X)$ in Sections~\ref{ss:rank14}, \ref{ss:rank10} and \ref{ss:rank8}. In the last column of Table~\ref{tab:order4}, we list the section where the reader can find the details. The details on the other exceptional cases are in Section~\ref{sec:order4}, \ref{sec:order8}, and \ref{sec:order12}, where we have described the grouping into equivalence classes. 
\end{enumerate}

Table~\ref{tab:order12} does not contain this last column, since for all but one of the equivalence classes the lattice $L_\cB$ has no overlattices (see Section~\ref{sec:order12}).

Observe that there are some cases where the group $\SL_W/J_W$ is not cyclic, namely 
 $x^4+y^4+z^4+w^4$ with weights $(1,1,1,1;4)$, $x^2+y^4+z^8+w^8$ with weights $(4,2,1,1;8)$ and $x^2+y^4+z^6+w^{12}$ with weights $(6,3,2,1;12)$.
In order to avoid confusion in recognizing the groups $G/J_W$ and doing the calculations, we 
detail the generators for each group $G/J_W$ for these polynomials.  
The last line on each table has $G=\SL_W$. 
\begin{itemize}
    \item $x^4+y^4+z^4+w^4$ with weights $(1,1,1,1;4)$:
\begin{center}
\begin{tabular}{ c|| c| c}
No. (order 4) & $G/J_W$ & Generators\\
\hline
3   & trivial& -\\
20  & $\ZZ/4\ZZ$ & $(\frac14,\frac34,0,0)$\\
23&$\ZZ/2\ZZ$&$(\frac12,\frac12,0,0)$\\
37&$\ZZ/2\ZZ\times\ZZ/2\ZZ$&$((\frac12,0,0,\frac12),(0,\frac12,0,\frac12))$\\
61&$\ZZ/2\ZZ\times\ZZ/4\ZZ$&$((\frac12,0,0\frac12),(\frac14,\frac34,0,0))$\\
63&$\ZZ/4\ZZ$&$(\frac14,0,\frac34,0)$\\
87&$\ZZ/4\ZZ\times\ZZ/4\ZZ$&$((\frac14,\frac34,0,0),(\frac14,0,\frac34,0))$\\
\end{tabular}
\end{center}
    
\item $x^2+y^4+z^8+w^8$ with weights $(4,2,1,1;4)$:
\begin{center}
\begin{tabular}{ c| c|| c| c}
No. (order 4) & No. (order 8) & $G/J_W$ & Generators\\
\hline
7 & 6 & trivial& -\\
11 & 15 & $\ZZ/2\ZZ$ &$(\frac12,0,\frac12,0)$\\
31 & 21 &$\ZZ/2\ZZ$&$(0,\frac12,\frac12,0)$\\
39 & 37 &$\ZZ/4\ZZ$&$(0,0,\frac18,\frac78)$\\
42 & 5 &$\ZZ/2\ZZ$&$(\frac12,\frac12,0,0)$\\
58 & 19 &$\ZZ/2\ZZ\times\ZZ/2\ZZ$&$((\frac12,\frac12,0,0),(0,0,\frac14,\frac34))$\\
77 & 24 &$\ZZ/4\ZZ$&$(0,\frac14,\frac34,0)$\\
83 & 33 &$\ZZ/2\ZZ\times\ZZ/4\ZZ$& $((\frac12,\frac12,0,0),(0,0,\frac18,\frac78))$\\
\end{tabular}
\end{center}

\item $x^2+y^4+z^6+w^{12}$ with weights $(6,3,2,1;12)$:
\begin{center}
\begin{tabular}{ c| c|| c| c}
No. (order 4) & No. (order 12) & $G/J_W$ & Generators\\
\hline
15 & 6 & trivial& -\\
33 & 25 & $\ZZ/2\ZZ$ &$(0,\frac12,\frac12,0)$\\
51 & 17 &$\ZZ/2\ZZ$&$(\frac12,0,\frac12,0)$\\
54 & 7 &$\ZZ/2\ZZ$&$(\frac12,\frac12,0,0)$\\
73 & 24 &$\ZZ/2\ZZ\times\ZZ/2\ZZ$&
$((\frac12,0,\frac12,0),(\frac12,\frac12,0,0))$\\
\end{tabular}
\end{center}

\end{itemize}

The first column in the previous tables refers to the numbering of Table~\ref{tab:order4}, and the second column of the last two tables refers to the numbering of Table~\ref{tab:order8} and Table~\ref{tab:order12}, respectively. 

\footnotesize

\begin{longtable}{ l c c c c c c c}

\caption{Table for $n=4$}
    \label{tab:order4}\\

  No. & $r_X$ & BHK Dual & Weights & Polynomial & $G/J_W$ & Form& Note\\ \hline
  \midrule
 \endfirsthead
 
  No. & $r_X$ & BHK Dual & Weights & Polynomial & $G/J_W$ & Form& Note
  \\\hline    \midrule
 \endhead
  1 & 1 & 85 & (1,1,1,1;4) & $x^4+y^3z+z^3w+yw^3$ & trivial & $w_{2,2}^1$& NOL
  \\
  \hline
  2 & 1 & 86 & (1,1,1,1;4) & $x^4+y^4+z^3w+zw^3$ & trivial & $w_{2,2}^1$ &
  \\
  \hline
  3 & 1 & 87 & (1,1,1,1;4) & $x^4+y^4+z^4+w^4$ & trivial & $w_{2,2}^1$ &
  \\
  \hline
  4 & 1 & 89 & (1,1,1,1;4) & $x^3w+y^4+z^4+w^4$ & trivial & $w_{2,2}^1$&
  \\
  \hline
  5 & 1 & 88 & (1,1,1,1;4) & $x^3z+y^4+z^3w+w^4$ & trivial & $w_{2,2}^1$&
  \\
  \hline \hline
  6* & 2 & 80 & (4,2,1,1;8) & $x^2+y^4+z^7w+zw^7$ & trivial & $u$ &7\\
  \hline
  7 & 2 & 83 & (4,2,1,1;8) & $x^2+y^4+z^8+w^8$ & trivial & $u$ &
  \\
  \hline
  8 & 2 & 81 & (4,2,1,1;8) & $x^2+y^4+xz^4+w^8$ & trivial & $u$&
  \\
  \hline 
  9 & 2 & 84 & (4,2,1,1;8) & $x^2+y^4+z^7w+w^8$ & trivial & $u$&
  \\
  \hline
  10 & 2 & 82 & (4,2,1,1;8) & $x^2+y^4+xz^4+zw^7$ & trivial & $u$&
  \\
  \hline \hline
  11 & 4 & 77 & (4,2,1,1;8) & $x^2+y^4+z^8+w^8$ & $\ZZ/2\ZZ$ & $w_{2,2}^{1}+w_{2,2}^5$ &
  \\
  \hline
  12* & 4 & 78 & (3,2,2,1;8) & $x^2z+y^4+z^4+w^8$ & trivial & $w_{2,2}^{1}+w_{2,2}^5$&19 \\
  \hline
  13 & 4 & 79 & (3,2,2,1;8) & $x^2z+y^4+z^4+xw^5$ & trivial & $w_{2,2}^{1}+w_{2,2}^5$ &
  \\
  \hline \hline
  14 & 5 & 72 & (6,3,2,1;12) & $x^2+y^4+xz^3+w^{12}$ & trivial & $u+w_{2,2}^5$ &
  \\
  \hline
  15 & 5 & 73 & (6,3,2,1;12) & $x^2+y^4+z^6+w^{12}$ & trivial & $u+w_{2,2}^5$ &
  \\
  \hline
  16* & 5 & 74 & (6,3,2,1;12) & $x^2+y^4+z^6+xw^6$ & trivial & $u+w_{2,2}^5$ &8\\
  \hline
  17 & 5 & 75 & (6,3,2,1;12) & $x^2+y^4+z^6+zw^{10}$ & trivial & $u+w_{2,2}^5$ &
  \\
  \hline
  18 & 5 & 76 & (6,3,2,1;12) & $x^2+y^4+xz^3+zw^{10}$ & trivial & $u+w_{2,2}^5$ &
  \\
  \hline \hline
  19 & 6 & 60 & (1,1,1,1;4) & $x^4+y^4+z^3w+zw^3$ & $\ZZ/2\ZZ$ & $v+v_2$&
  \\
  \hline
  20 & 6 & 63 & (1,1,1,1;4) & $x^4+y^4+z^4+w^4$ & $\ZZ/4\ZZ$ & $v+v_2$ &
  \\
  \hline
  21* & 6 & 64 & (4,3,3,2;12) & $x^3+y^4+z^4+xw^4$ & trivial & $v+v_2$&2 \\
  \hline
  22 & 6 & 65 & (4,3,3,2;12) & $x^3+y^4+z^4+w^6$ & trivial & $v+v_2$ &
  \\
  \hline
  23 & 6 & 61 & (1,1,1,1;4) & $x^4+y^4+z^4+w^4$ & $\ZZ/2\ZZ$ & $v+v_2$ &
  \\
  \hline
  24 & 6 & 62 & (1,1,1,1;4) & $x^3w+y^4+z^4+w^4$ & $\ZZ/2\ZZ$ & $v+v_2$&
  \\
  \hline \hline
  25* & 6 & 66 & (10,5,4,1;20) & $x^2+y^4+z^5+w^{20}$ & trivial & $u+v$ & 9
  \\
  \hline
  26 & 6 & 69 & (8,4,3,1;16) & $x^2+y^4+z^5w+w^{16}$ & trivial & $u+v$ &
  \\
  \hline
  27 & 6 & 67 & (10,5,4,1;20) & $x^2+y^4+z^5+xw^{10}$ & trivial & $u+v$ &
  \\
  \hline
  28 & 6 & 71 & (8,4,3,1;16) & $x^2+y^4+z^5w+xw^8$ & trivial & $u+v$ &
  \\
  \hline
  29 & 6 & 68 & (10,5,4,1;20) & $x^2+y^4+z^5+zw^{16}$ & trivial & $u+v$ &
  \\
  \hline
  30 & 6 & 70 & (8,4,3,1;16) & $x^2+y^4+z^5w+zw^{13}$ & trivial & $u+v$&
  \\
  \hline \hline
  $31^E$ & 8 & 58 & (4,2,1,1;8) & $x^2+y^4+z^8+w^8$ & $\ZZ/2\ZZ$ & $v+w_{2,2}^1+w_{2,2}^5$ & S.\ref{ss:rank8}
  \\
  \hline
  32 & 8 & 59 & (4,2,1,1;8) & $x^2+y^4+xz^4+w^8$ & $\ZZ/2\ZZ$ & $v+w_{2,2}^1+w_{2,2}^5$&
  \\
  \hline \hline
  33 & 9 & 54 & (6,3,2,1;12) & $x^2+y^4+z^6+w^{12}$ & $\ZZ/2\ZZ$ & $u+v+w_{2,2}^5$ &
  \\
  \hline
  34* & 9 & 57 & (10,5,3,2;20) & $x^2+y^4+z^6w+w^{10}$ & trivial & $u+v+w_{2,2}^5$ &36\\
  \hline
  35 & 9 & 55 & (6,3,2,1;12) & $x^2+y^4+z^6+xw^6$ & $\ZZ/2\ZZ$ & $u+v+w_{2,2}^5$ &
  \\
  \hline
  36 & 9 & 56 & (10,5,3,2;20) & $x^2+y^4+z^6w+xw^5$ & trivial & $u+v+w_{2,2}^5$&
  \\
  \hline \hline

  $37^E$ & 10 & self & (1,1,1,1;4) & $x^4+y^4+z^4+w^4$ & $\ZZ/2\ZZ \times \ZZ/2\ZZ$ & $w_{2,2}^1\oplus (w_{2,2}^5)^3$ & S.\ref{ss:rank10}
  \\
  \hline \hline
  38 & 10 & 41 & (4,2,1,1;8) & $x^2+y^4+z^7w+zw^7$ & $\ZZ/3\ZZ$ & $u^3$&
  \\
  \hline
  39 & 10 & 42 & (4,2,1,1;8) & $x^2+y^4+z^8+w^8$ & $\ZZ/4\ZZ$ & $u^3$ &
  \\
  \hline
  40* & 10 & 43 & (14,7,4,3;28) & $x^2+y^4+z^7+zw^8$ & trivial & $u^3$ &35\\
  \hline \hline
  41 & 10 & 38 & (4,2,1,1;8) & $x^2+y^4+z^7w+zw^7$ & $\ZZ/2\ZZ$ & $u^3$ &
  \\
  \hline
  $42^E$ & 10 & 39 & (4,2,1,1;8) & $x^2+y^4+z^8+w^8$ & $\ZZ/2\ZZ$ & $u^3$ &S:\ref{sec:order4}
  \\
  \hline
  43 & 10 & 40 & (4,2,1,1;8) & $x^2+y^4+z^7w+w^8$ & $\ZZ/2\ZZ$ & $u^3$&
  \\
  \hline \hline
  44* & 10 & 52 & (8,7,6,3;24) & $x^3+y^3w+z^4+w^8$ & trivial & $v_2$ &16\\
  \hline
  45 & 10 & 49 & (4,4,3,1;12) & $x^3+y^3+z^4+w^{12}$ & trivial & $v_2$ &
  \\
  \hline
  46 & 10 & self & (4,4,3,1;12) & $x^2y+xy^2+z^4+w^{12}$ & trivial & $v_2$&
  \\
  \hline 
  47 & 10 & 50 & (6,3,2,1;12) & $x^2+y^4+xz^3+w^{12}$ & $\ZZ/2\ZZ$ & $v_2$ &
  \\
  \hline
  48 & 10 & 53 & (7,4,3,2;16) & $x^2w+y^4+xz^3+w^8$ & trivial & $v_2$ &
  \\
  \hline
  49 & 10 & 45 & (4,4,3,1;12) & $x^3+y^3+z^4+w^{12}$ & $\ZZ/3\ZZ$ & $v_2$ &
  \\
  \hline 
  50 & 10 & 47 & (4,4,3,1;12) & $x^2y+y^3+z^4+w^{12}$ & trivial & $v_2$ &
  \\
  \hline
  51 & 10 & self & (6,3,2,1;12) & $x^2+y^4+z^6+w^{12}$ & $\ZZ/2\ZZ$ & $v_2$ &
  \\
  \hline
  52 & 10 & 44 & (4,4,3,1;12) & $x^3+y^3+z^4+yw^8$ & trivial & $v_2$ &
  \\
  \hline 
  53 & 10 & 48 & (4,4,3,1;12) & $x^3+xy^2+z^4+yw^8$ & trivial & $v_2$ &
  \\
  \hline \hline
  54 & 11 & 33 & (6,3,2,1;12) & $x^2+y^4+z^6+w^{12}$ & $\ZZ/2\ZZ$ & $u+v+w_{2,2}^{-5}$ &
  \\
  \hline
  55* & 11 & 35 & (5,3,2,2;12) & $x^2w+y^4+z^6+w^6$ & trivial & $u+v+w_{2,2}^{-5}$ &23\\
  \hline
  56 & 11 & 36 & (5,3,2,2;12) & $x^2w+y^4+z^6+zw^5$ & trivial & $u+v+w_{2,2}^{-5}$&
  \\
  \hline
  57 & 11 & 34 & (6,3,2,1;12) & $x^2+y^4+z^6+zw^{10}$ & $\ZZ/2\ZZ$ & $u+v+w_{2,2}^{-5}$ &
  \\
  \hline \hline
  $58^E$ & 12 & 31 & (4,2,1,1;8) & $x^2+y^4+z^8+w^8$ & $\ZZ/2\ZZ \times \ZZ/2\ZZ$ & $v+w_{2,2}^{-1}+w_{2,2}^{-5}$ &S:\ref{sec:order4}
  \\
  \hline
  59 & 12 & 32 & (3,2,1,1;8) & $x^2z+y^4+z^4+w^8$ & $\ZZ/2\ZZ$ & $v+w_{2,2}^{-1}+w_{2,2}^{-5}$ &
  \\
  \hline \hline
  60 & 14 & 19 & (1,1,1,1;4) & $x^4+y^4+z^3w+zw^3$ & $\ZZ/4\ZZ$ & $v+v_2$ &
  \\
  \hline
 61 & 14 & 23 & (1,1,1,1;4) & $x^4+y^4+z^4+w^4$ & $\ZZ/2\ZZ \times \ZZ/4\ZZ$ & $v+v_2$ &
 \\
 \hline
 62 & 14 & 24 & (4,3,3,2;12) & $x^3+y^4+z^4+xw^4$ & $\ZZ/2\ZZ$ & $v+v_2$ &
 \\
 \hline
 $63^E$ & 14 & 20 & (1,1,1,1;4) & $x^4+y^4+z^4+w^4$ & $\ZZ/4\ZZ$ & $v+v_2$ & S.\ref{ss:rank14}
 \\
 \hline
 64 & 14 & 21 & (1,1,1,1;4) & $x^3w+y^4+z^4+w^4$ & $\ZZ/4\ZZ$ & $v+v_2$ &
 \\
 \hline
 65 & 14 & 22 & (4,3,3,2;12) & $x^3+y^4+z^4+w^6$ & $\ZZ/2\ZZ$ & $v+v_2$ &
 \\
 \hline \hline
 66 & 14 & 25 & (10,5,4,1;20) & $x^2+y^4+z^5+w^{20}$ & $\ZZ/2\ZZ$ & $u+v$ &
 \\
 \hline
 67* & 14 & 27 & (9,5,4,2;20) & $x^2w+y^4+z^5+w^{10}$ & trivial & $u+v$ &26\\
 \hline
 68 & 14 & 29 & (8,4,3,1;16) & $x^2+y^4+z^5w+w^{16}$ & $\ZZ/2\ZZ$ & $u+v$ &
 \\
 \hline 
 69 & 14 & 26 & (10,5,4,1;20) & $x^2+y^4+z^5+zw^{16}$ & $\ZZ/2\ZZ$ & $u+v$ &
 \\
 \hline
 70 & 14 & 30 & (8,4,3,1;16) & $x^2+y^4+z^5w+zw^{13}$ & $\ZZ/2\ZZ$ & $u+v$ &
 \\
 \hline
 71 & 14 & 28 & (9,5,4,2;20) & $x^2w+y^4+z^5+zw^8$ & trivial & $u+v$ &
 \\
 \hline \hline

 72 & 15 & 14 & (4,4,3,1;12) & $x^2y+y^3+z^4+w^{12}$ & $\ZZ/2\ZZ$ & $u+w_{2,2}^{-5}$ &
 \\
 \hline
 73 & 15 & 15 & (6,3,2,1;12) & $x^2+y^4+z^6+w^{12}$ & $\ZZ/2\ZZ \times \ZZ/2\ZZ$ & $u+w_{2,2}^{-5}$ &
 \\
 \hline
 74 & 15 & 16 & (5,3,2,2;12) & $x^2w+y^4+z^6+w^6$ & $\ZZ/2\ZZ$ & $u+w_{2,2}^{-5}$ &
 \\
 \hline
 75 & 15 & 17 & (10,5,3,2;20) & $x^2+y^4+z^6w+w^{10}$ & $\ZZ/2\ZZ$ & $u+w_{2,2}^{-5}$&
 \\
 \hline
 76* & 15 & 18 & (7,6,5,2;20) & $x^2y+y^3w+z^4+w^{10}$ & trivial & $u+w_{2,2}^{-5}$&55 \\
 \hline \hline
 77 & 16 & 11 & (4,2,1,1;8) & $x^2+y^4+z^8+w^8$ & $\ZZ/4\ZZ$ & $w_{2,2}^{-1}+w_{2,2}^{-5}$&
 \\
 \hline
 78 & 16 & 12 & (4,2,1,1;8) & $x^2+y^4+xz^4+w^8$ & $\ZZ/4\ZZ$ & $w_{2,2}^{-1}+w_{2,2}^{-5}$ &
 \\
 \hline
 79* & 16 & 13 & (8,5,4,3;20) & $x^2z+y^4+z^5+xw^4$ & trivial & $w_{2,2}^{-1}+w_{2,2}^{-5}$&62 \\
 \hline \hline
 80 & 18 & 6 & (4,2,1,1;8) & $x^2+y^4+z^7w+zw^7$ & $\ZZ/6\ZZ$ & $u$ &
 \\
 \hline
 81 & 18 & 8 & (3,2,2,1;8) & $x^2z+y^4+z^4+w^8$ & $\ZZ/4\ZZ$ & $u$ &
 \\
 \hline
 82* & 18 & 10 & (11,7,6,4;28) & $x^2z+y^4+z^4w+w^7$ & trivial & $u$ &61\\
 \hline
 83 & 18 & 7 & (4,2,1,1;8) & $x^2+y^4+z^8+w^8$ & $\ZZ/2\ZZ \times \ZZ/4\ZZ$ & $u$&
 \\
 \hline
 84 & 18 & 9 & (14,7,4,3;28) & $x^2+y^4+z^7+zw^8$ & $\ZZ/2\ZZ$ & $u$ &
 \\
 \hline \hline
 85 & 19 & 1 & (1,1,1,1;4) & $x^4+y^3z+z^3w+yw^3$ & $\ZZ/7\ZZ$ & $w_{2,2}^{-1}$ &
 \\
 \hline
 86 & 19 & 2 & (1,1,1,1;4) & $x^4+y^4+z^3w+zw^3$ & $\ZZ/8\ZZ$ & $w_{2,2}^{-1}$ &
 \\
 \hline
 87 & 19 & 3 & (1,1,1,1;4) & $x^4+y^4+z^4+w^4$ & $\ZZ/4\ZZ \times \ZZ/4\ZZ$ & $w_{2,2}^{-1}$ &
 \\
 \hline
 88* & 19 & 5 & (12,9,8,7;36) & $x^3+y^4+xz^3+zw^4$ & trivial & $w_{2,2}^{-1}$ &52\\
 \hline
 89 & 19 & 4 & (4,3,3,2;12) & $x^3+y^4+z^4+xw^4$ & $\ZZ/4\ZZ$ & $w_{2,2}^{-1}$ &
 \\
 \hline
 
 \end{longtable}

\begin{longtable}{ l c c c c c c c}
\caption{Table for $n=8$}
    \label{tab:order8}\\
  No. & $r_X$ & BHK Dual & Weights & Polynomial & $G/J_W$ & Form& Note\\ \hline
  \midrule
 \endfirsthead

  No & $r_X$ & BHK Dual & Weights & Polynomial & $G/J_W$ & Form& Note
  \\\hline    \midrule
 \endhead
 1* & 3 & 36 & (4,2,1,1;8) & $x^2+y^4+xz^4+w^8$ & trivial & $w_{2,2}^{-1}$ &7 \\
 \hline
 2 & 3 & 34 & (4,2,1,1;8) & $x^2+y^4+yz^6+w^8$ & trivial & $w_{2,2}^{-1}$&
 \\
 \hline
 3 & 3 & 35 & (4,2,1,1;8) & $x^2+xy^2+yz^6+w^8$ & trivial & $w_{2,2}^{-1}$ &
 \\
 \hline
 4 & 3 & 32 & (4,2,1,1;8) & $x^2+xy^2+z^8+w^8$ & trivial & $w_{2,2}^{-1}$ &
 \\
 \hline
 5 & 3 & 37 & (4,2,1,1;8) & $x^2+y^4+z^8+w^8$ & $\ZZ/2\ZZ$ & $w_{2,2}^{-1}$ &
 \\
 \hline
 6 & 3 & 33 & (4,2,1,1;8) & $x^2+y^4+z^8+w^8$ & trivial & $w_{2,2}^{-1}$ &
 \\
 \hline \hline
 7 & 6 & 29 & (12,8,3,1;24) & $x^2+y^3+z^8+xw^{12}$ & trivial & $v$ &
 \\
 \hline
 8* & 6 & 27 & (8,5,2,1;16) & $x^2+y^3w+z^8+xw^8$ & trivial & $v$ &44\\
 \hline
 9 & 6 & 30 & (12,8,3,1,24) & $x^2+y^3+z^8+yw^{16}$ & trivial & $v$ &
 \\
 \hline
 10 & 6 & 28 & (8,5,2,1;16) & $x^2+y^3w+z^8+yw^{11}$ & trivial & $v$ &
 \\
 \hline
 11 & 6 & 31 & (12,8,3,1;24) & $x^2+y^3+z^8+w^{24}$ & trivial & $v$ &
 \\
 \hline
 12 & 6 & 26 & (8,5,2,1;16) & $x^2+y^3w+z^8+w^{16}$ & trivial & $v$ &
 \\ 
 \hline \hline
 13 & 7 & 25 & (4,2,1,1;8) & $x^2+y^4+yz^6+w^8$ & $\ZZ/2\ZZ$ & $v+w_{2,3}^{-1}$ &
 \\
 \hline 
 14* & 7 & 23 & (3,2,2,1;8) & $x^2z+y^4+yz^3+w^8$ & trivial & $v+w_{2,3}^{-1}$ &19\\
 \hline 
 15 & 7 & 24 & (4,2,1,1;8) & $x^2+y^4+z^8+w^8$ & $\ZZ/2\ZZ$ & $v+w_{2,3}^{-1}$ &
 \\
 \hline
 16 & 7 & 22 & (3,2,2,1;8) & $x^2z+y^4+z^4+w^8$ & trivial & $v+w_{2,3}^{-1}$ &
 \\
 \hline \hline
 17 & 10 & 20 & (4,2,1,1;8) & $x^2+y^4+xz^4+w^8$ & $\ZZ/2\ZZ$ & $w_{2,2}^{-1}+w_{2,3}^{1}$ &
 \\
 \hline
 18 & 10 & 18 & (4,2,1,1;8) & $x^2+xy^2+z^8+w^8$ & $\ZZ/2\ZZ$ & $w_{2,2}^{-1}+w_{2,3}^{1}$&
 \\
 \hline
 $19^E$ & 10 & 21 & (4,2,1,1;8) & $x^2+y^4+z^8+w^8$ & $\ZZ/2\ZZ \times \ZZ/2\ZZ$ & $w_{2,2}^{-1}+w_{2,3}^{1}$ &S. \ref{sec:order8}
 \\
 \hline
 20 & 10 & 17 & (3,2,2,1;8) & $x^2z+y^4+z^4+w^8$ & $\ZZ/2\ZZ$ & $w_{2,2}^{-1}+w_{2,3}^{1}$ &
 \\
 \hline
 21 & 10 & 19 & (4,2,1,1;8) & $x^2+y^4+z^8+w^8$ & $\ZZ/2\ZZ$ & $w_{2,2}^{-1}+w_{2,3}^{1}$ &
 \\
 \hline \hline
 22 & 13 & 16 & (4,2,1,1;8) & $x^2+y^4+xz^4+w^8$ & $\ZZ/4\ZZ$ & $v+w_{2,3}^1$ &
 \\
 \hline
 23* & 13 & 14 & (12,5,4,3;24) & $x^2+y^4z+xz^3+w^8$ & trivial & $v+w_{2,3}^1$ &31\\
 \hline
 24 & 13 & 15 & (4,2,1,1;8) & $x^2+y^4+z^8+w^8$ & $\ZZ/4\ZZ$ & $v+w_{2,3}^1$ &
 \\
 \hline
 25 & 13 & 13 & (12,5,4,3;24) & $x^2+y^4z+z^6+w^8$ & trivial & $v+w_{2,3}^1$ &
 \\
 \hline \hline
 26 & 14 & 12 & (12,8,3,1;24) & $x^2+y^3+z^8+yw^{16}$ & $\ZZ/2\ZZ$ & $v$ &
 \\
 \hline
 27* & 14 & 8 & (11,8,3,2;24) & $x^2w+y^3+z^8+yw^8$ & trivial & $v$ &27\\
 \hline
 28 & 14 & 10 & (8,5,2,1;16) & $x^2+y^3w+z^8+yw^{11}$ & $\ZZ/2\ZZ$ & $v$ &
 \\
 \hline 
 29 & 14 & 7 & (11,8,3,2;24) & $x^2w+y^3+z^8+w^{12}$ & trivial & $v$ &
 \\
 \hline
 30 & 14 & 9 & (8,5,2,1;16) & $x^2+y^3w+z^8+w^{16}$ & $\ZZ/2\ZZ$ & $v$ &
 \\
 \hline
 31 & 14 & 11 & (12,8,3,1;24) & $x^2+y^3+z^8+w^{24}$ & $\ZZ/2\ZZ$ & $v$ &
 \\
 \hline \hline
 32 & 17 & 4 & (4,2,1,1;8) & $x^2+xy^2+z^8+w^8$ & $\ZZ/4\ZZ$ & $w_{2,2}^1$&
 \\
 \hline
 33 & 17 & 6 & (4,2,1,1;8) & $x^2+y^4+z^8+w^8$ & $\ZZ/2\ZZ \times \ZZ/4\ZZ$ & $w_{2,2}^1$ &
 \\
 \hline
 34 & 17 & 2 & (12,5,4,3;24) & $x^2+y^4z+z^6+w^8$ & $\ZZ/2\ZZ$ & $w_{2,2}^1$&
 \\
 \hline
 35* & 17 & 3 & (10,7,4,3;24) & $x^2z+xy^2+z^6+w^8$ & trivial & $w_{2,2}^1$ &64\\
 \hline
 36 & 17 & 1 & (3,2,2,1;8) & $x^2z+y^4+z^4+w^8$ & $\ZZ/4\ZZ$ & $w_{2,2}^1$ &
 \\
 \hline
 37 & 17 & 5 & (4,2,1,1;8) & $x^2+y^4+z^8+w^8$ & $\ZZ/4\ZZ$ & $w_{2,2}^1$ &
 \\
 \hline
 \end{longtable}

 \begin{longtable}{ c c c c c c c}
 \caption{Table for $n=12$} \label{tab:order12}\\
 
  No. & $r_X$ & BHK Dual & Weights & Polynomial & $G/J_W$ & Form
  \\ \hline
  \midrule
 \endfirsthead

  No. & $r_X$ & BHK Dual & Weights & Polynomial & $G/J_W$ & Form 
  \\\hline    \midrule
 \endhead
 1 & 2 & 26 & (6,4,1,1;12) & $x^2+y^3+xz^6+w^{12}$ & trivial & trivial 
 \\
 \hline
 2 & 2 & 27 & (6,4,1,1;12) & $x^2+y^3+yz^8+w^{12}$ & trivial & trivial 
 \\
 \hline
 3 & 2 & 28 & (6,4,1,1;12) & $x^2+y^3+z^{12}+w^{12}$ & trivial & trivial 
 \\
 \hline \hline
 4 & 7 & 22 & (6,3,2,1;12) & $x^2+y^4+xz^3+w^{12}$ & trivial & $w_{2,2}^1+w_{3,1}^{-1}$
 \\
 \hline
 5 & 7 & 23 & (6,3,2,1;12) & $x^2+xy^2+z^6+w^{12}$ & trivial & $w_{2,2}^1+w_{3,1}^{-1}$ 
 \\
 \hline
 6 & 7 & 24 & (6,3,2,1;12) & $x^2+y^4+z^6+w^{12}$ & trivial & $w_{2,2}^1+w_{3,1}^{-1}$ 
 \\
 \hline
 7 & 7 & 25 & (6,3,2,1;12) & $x^2+y^4+z^6+w^{12}$ & $\ZZ/2\ZZ$ & $w_{2,2}^1+w_{3,1}^{-1}$ 
 \\
 \hline \hline
 8 & 8 & 20 & (5,4,2,1;12) & $x^2z+y^3+z^6+w^{12}$ & trivial & $v+w_{3,1}^1$ 
 \\
 \hline
 9 & 8 & 19 & (5,4,2,1;12) & $x^2z+y^3+yz^4+w^{12}$ & trivial & $v+w_{3,1}^1$ 
 \\
 \hline
 10 & 8 & 18 & (6,4,1,1;12) & $x^2+y^3+yz^8+w^{12}$ & $\ZZ/2\ZZ$ & $v+w_{3,1}^1$ 
 \\
 \hline
 11 & 8 & 21 & (6,4,1,1;12) & $x^2+y^3+z^{12}+w^{12}$ & $\ZZ/2\ZZ$ & $v+w_{3,1}^1$ 
 \\
 \hline \hline
 12 & 10 & 16 & (4,4,3,1;12) & $x^2y+y^3+z^4+w^{12}$ & trivial & $v_2$ 
 \\
 \hline
 $13^E$ & 10 & 14 & (4,4,3,1;12) & $x^3+y^3+z^4+w^{12}$ & trivial & $v_2$ 
 \\
 \hline
 14 & 10 & 13 & (4,4,3,1;12) & $x^3+y^3+z^4+w^{12}$ & $\ZZ/3\ZZ$ & $v_2$ 
 \\
 \hline
 15 & 10 & self & (4,4,3,1;12) & $x^2y+xy^2+z^4+w^{12}$ & trivial & $v_2$ 
 \\
 \hline
 16 & 10 & 12 & (6,3,2,1;12) & $x^2+y^4+xz^3+w^{12}$ & $\ZZ/2\ZZ$ & $v_2$ 
 \\
 \hline
 17 & 10 & self & (6,3,2,1;12) & $x^2+y^4+z^6+w^{12}$ & $\ZZ/2\ZZ$ & $v_2$ 
 \\
 \hline \hline
 18 & 12 & 10 & (12,7,3,2;24) & $x^2+y^3z+z^8+w^{12}$ & trivial & $v+w_{3,1}^{-1}$ 
 \\
 \hline
 19 & 12 & 9 & (12,7,3,2;24) & $x^2+y^3z+xz^4+w^{12}$ & trivial & $v+w_{3,1}^{-1}$ 
 \\
 \hline
 20 & 12 & 8 & (6,4,1,1;12) & $x^2+y^3+xz^6+w^{12}$ & $\ZZ/3\ZZ$ & $v+w_{3,1}^{-1}$ 
 \\
 \hline
 21 & 12 & 11 & (6,4,1,1;12) & $x^2+y^3+z^{12}+w^{12}$ & $\ZZ/3\ZZ$ & $v+w_{3,1}^{-1}$ 
 \\
 \hline \hline
 22 & 13 & 4 & (4,4,3,1;12) & $x^2y+y^3+z^4+w^{12}$ & $\ZZ/2\ZZ$ & $w_{2,2}^{-1}+w_{3,1}^{1}$ 
 \\
 \hline
 23 & 13 & 5 & (6,3,2,1;12) & $x^2+xy^2+z^6+w^{12}$ & $\ZZ/2\ZZ$ & $w_{2,2}^{-1}+w_{3,1}^{1}$ 
 \\
 \hline
 24 & 13 & 6 & (6,3,2,1;12) & $x^2+y^4+z^6+w^{12}$ & $\ZZ/2\ZZ \times \ZZ/2\ZZ$ & $w_{2,2}^{-1}+w_{3,1}^{1}$ 
 \\
 \hline
 25 & 13 & 7 & (6,3,2,1;12) & $x^2+y^4+z^6+w^{12}$ & $\ZZ/2\ZZ$ & $w_{2,2}^{-1}+w_{3,1}^{1}$
 \\
 \hline \hline
 26 & 18 & 1 & (5,4,2,1;12) & $x^2z+y^3+z^6+w^{12}$ & $\ZZ/3\ZZ$ & trivial 
 \\
 \hline
 27 & 18 & 2 & (12,7,3,2;24) & $x^2+y^3z+z^8+w^{12}$ & $\ZZ/2\ZZ$ & trivial 
 \\
 \hline
 28 & 18 & 3 & (6,4,1,1;12) & $x^2+y^3+z^{12}+w^{12}$ & $\ZZ/6\ZZ$ & trivial 
 \\
 \hline
 \end{longtable}

%% file: main.bbl
\begin{thebibliography}{10}

\bibitem{ABS}
Michela Artebani, Samuel Boissi{\`e}re, and Alessandra Sarti.
\newblock The {B}erglund-{H}\"ubsch-{C}hiodo-{R}uan mirror symmetry for {K}3
  surfaces.
\newblock {\em J. Math. Pures Appl. (9)}, 102(4):758--781, 2014.

\bibitem{surfaces}
Wolf~P. Barth, Klaus Hulek, Chris A.~M. Peters, and Antonius Van~de Ven.
\newblock {\em Compact complex surfaces}, volume~4 of {\em Ergebnisse der
  Mathematik und ihrer Grenzgebiete. 3. Folge. A Series of Modern Surveys in
  Mathematics}.
\newblock Springer-Verlag, Berlin, second edition, 2004.

\bibitem{batyrev}
Victor~V. Batyrev.
\newblock Dual polyhedra and mirror symmetry for {C}alabi-{Y}au hypersurfaces
  in toric varieties.
\newblock {\em J. Algebraic Geom.}, 3(3):493--535, 1994.

\bibitem{belcastro}
Sarah-Marie {Belcastro}.
\newblock {\em {Picard Lattices of Families of {K}3 Surfaces}}.
\newblock PhD thesis, University of Michigan, 1997.

\bibitem{berghenn}
Per Berglund and M{\aa}ns Henningson.
\newblock Landau-{G}inzburg orbifolds, mirror symmetry and the elliptic genus.
\newblock {\em Nuclear Phys. B}, 433(2):311--332, 1995.

\bibitem{berghub}
Per Berglund and Tristan H{\"u}bsch.
\newblock A generalized construction of mirror manifolds.
\newblock {\em Nuclear Phys. B}, 393(1-2):377--391, 1993.

\bibitem{magma}
Wieb Bosma, John Cannon, and Catherine Playoust.
\newblock The {M}agma algebra system. {I}. {T}he user language.
\newblock {\em J. Symbolic Comput.}, 24(3-4):235--265, 1997.
\newblock Computational algebra and number theory (London, 1993).

\bibitem{bott}
Christopher~James {B}ott.
\newblock {\em {Mirror symmetry for algebraic {K}3 surfaces with non-symplectic
  automorphism}}.
\newblock PhD thesis, Brigham Young University, 2018.

\bibitem{candelas}
Philip Candelas, Xenia~C. de~la Ossa, Paul~S. Green, and Linda Parkes.
\newblock A pair of {C}alabi-{Y}au manifolds as an exactly soluble
  superconformal theory.
\newblock {\em Nuclear Phys. B}, 359(1):21--74, 1991.

\bibitem{CIR}
Alessandro Chiodo, Hiroshi Iritani, and Yongbin Ruan.
\newblock Landau-{G}inzburg/{C}alabi-{Y}au correspondence, global mirror
  symmetry and {O}rlov equivalence.
\newblock {\em Publ. Math. Inst. Hautes \'{E}tudes Sci.}, 119:127--216, 2014.

\bibitem{ChKV}
Alessandro Chiodo, Elana Kalashnikov, and Davide Veniani.
\newblock Semi-{C}alabi--{Y}au orbifolds and mirror pairs.
\newblock arXiv:1509.06685v2.

\bibitem{ChR}
Alessandro Chiodo and Yongbin Ruan.
\newblock A global mirror symmetry framework for the
  {L}andau-{G}inzburg/{C}alabi-{Y}au correspondence.
\newblock {\em Ann. Inst. Fourier (Grenoble)}, 61(7):2803--2864, 2011.

\bibitem{BHCR}
Alessandro Chiodo and Yongbin Ruan.
\newblock L{G}/{CY} correspondence: the state space isomorphism.
\newblock {\em Adv. Math.}, 227(6):2157--2188, 2011.

\bibitem{clader}
Emily Clader.
\newblock Landau-{G}inzburg/{C}alabi-{Y}au correspondence for the complete
  intersections {$X_{3,3}$} and {$X_{2,2,2,2}$}.
\newblock {\em Adv. Math.}, 307:1--52, 2017.

\bibitem{CLPS}
Paola Comparin, Christopher Lyons, Nathan Priddis, and Rachel Suggs.
\newblock The mirror symmetry of {K}3 surfaces with non-symplectic
  automorphisms of prime order.
\newblock {\em Adv. Theor. Math. Phys.}, 18(6):1335--1368, 2014.

\bibitem{CP}
Paola Comparin and Nathan Priddis.
\newblock {BHK} mirror symmetry for {K}3 surfaces with non-symplectic
  automorphism.
\newblock arxiv:1704.00354.

\bibitem{corti-golyshev}
Alessio Corti and Vasily Golyshev.
\newblock Hypergeometric equations and weighted projective spaces.
\newblock {\em Sci. China Math.}, 54(8):1577--1590, 2011.

\bibitem{degtyarev}
Alex Degtyarev.
\newblock Lines generate the {P}icard groups of certain {F}ermat surfaces.
\newblock {\em J. Number Theory}, 147:454--477, 2015.

\bibitem{dolgachev}
I.~V. Dolgachev.
\newblock Mirror symmetry for lattice polarized {$K3$} surfaces.
\newblock {\em J. Math. Sci.}, 81(3):2599--2630, 1996.
\newblock Algebraic geometry, 4.

\bibitem{givental}
Alexander Givental.
\newblock Elliptic {G}romov-{W}itten invariants and the generalized mirror
  conjecture.
\newblock In {\em Integrable systems and algebraic geometry ({K}obe/{K}yoto,
  1997)}, pages 107--155. World Sci. Publ., River Edge, NJ, 1998.

\bibitem{guere}
J\'{e}r\'{e}my Gu\'{e}r\'{e}.
\newblock A {L}andau-{G}inzburg mirror theorem without concavity.
\newblock {\em Duke Math. J.}, 165(13):2461--2527, 2016.

\bibitem{HV}
Kentaro Hori and Cumrun Vafa.
\newblock Mirror symmetry.
\newblock arXiv:hep-th/0002222v3.

\bibitem{huybrechts}
Daniel Huybrechts.
\newblock {\em Lectures on {K}3 surfaces}, volume 158 of {\em Cambridge Studies
  in Advanced Mathematics}.
\newblock Cambridge University Press, Cambridge, 2016.

\bibitem{kelly}
Tyler Kelly.
\newblock Berglund-{H}\"ubsch-{K}rawitz mirrors via {S}hioda maps.
\newblock {\em Adv. Theor. Math. Phys.}, 17:1425--1449, 2013.

\bibitem{krawitz}
Marc {Krawitz}.
\newblock {\em {FJRW rings and Landau-Ginzburg Mirror Symmetry}}.
\newblock PhD thesis, University of Michigan, 2010.

\bibitem{kk}
Maximilian Kreuzer.
\newblock The mirror map for invertible {LG} models.
\newblock {\em Phys. Lett. B}, 328(3-4):312--318, 1994.

\bibitem{KrSk}
Maximilian Kreuzer and Harald Skarke.
\newblock On the classification of quasihomogeneous functions.
\newblock {\em Comm. Math. Phys.}, 150(1):137--147, 1992.

\bibitem{nikulin}
V.~V. Nikulin.
\newblock Integer symmetric bilinear forms and some of their geometric
  applications.
\newblock {\em Izv. Akad. Nauk SSSR Ser. Mat.}, 43(1):111--177, 238, 1979.

\bibitem{PSh}
Nathan Priddis and Mark Shoemaker.
\newblock A {L}andau-{G}inzburg/{C}alabi-{Y}au correspondence for the mirror
  quintic.
\newblock {\em Ann. Inst. Fourier (Grenoble)}, 66(3):1045--1091, 2016.

\bibitem{SSvL}
Matthias Sch\"{u}tt, Tetsuji Shioda, and Ronald van Luijk.
\newblock Lines on {F}ermat surfaces.
\newblock {\em J. Number Theory}, 130(9):1939--1963, 2010.

\bibitem{shoemaker}
Mark Shoemaker.
\newblock Birationality of {B}erglund-{H}\"ubsch-{K}rawitz mirrors.
\newblock {\em Comm. Math. Phys.}, 331(2):417--429, 2014.

\bibitem{yonemura}
Takashi Yonemura.
\newblock Hypersurface simple {K$3$} singularities.
\newblock {\em Tohoku Math. J. (2)}, 42(3):351--380, 1990.

\end{thebibliography}
